\documentclass{article}%
\usepackage{amsfonts}
\usepackage{amsmath}
\usepackage{amssymb}
\usepackage{graphicx, xcolor}
\usepackage{graphicx}%
\providecommand{\U}[1]{\protect\rule{.1in}{.1in}}
\newtheorem{theorem}{Theorem}

\newtheorem{corollary}[theorem]{Corollary}

\newtheorem{definition}[theorem]{Definition}
\newtheorem{example}[theorem]{Example}

\newtheorem{lemma}[theorem]{Lemma}
\newtheorem{notation}[theorem]{Notation}

\newtheorem{proposition}[theorem]{Proposition}
\newtheorem{remark}[theorem]{Remark}

\newenvironment{proof}[1][Proof]{\noindent\textbf{#1.} }{\ \rule{0.5em}{0.5em}}
\numberwithin{equation}{section}
\numberwithin{theorem}{section}
\newenvironment{sergiorev}{(\color{blue}}{\color{black})}
\newcommand{\bsr}{\begin{sergiorev}}
\newcommand{\esr}{\end{sergiorev}}
\begin{document}

\title{Fundamental solutions for Kolmogorov-Fokker-Planck operators with
time-depending measurable coefficients}
\author{Marco Bramanti,\ Sergio Polidoro}
\maketitle

\begin{abstract}
We consider a Kolmogorov-Fokker-Planck operator of the kind:%
\[
\mathcal{L}u=\sum_{i,j=1}^{q}a_{ij}\left(  t\right)  \partial_{x_{i}x_{j}}%
^{2}u+\sum_{k,j=1}^{N}b_{jk}x_{k}\partial_{x_{j}}u-\partial_{t}u,\qquad
(x,t)\in\mathbb{R}^{N+1}%
\]
where $\left\{  a_{ij}\left(  t\right)  \right\}  _{i,j=1}^{q}$ is a symmetric
uniformly positive matrix on $\mathbb{R}^{q}$, $q\leq N$, of bounded
measurable coefficients defined for $t\in\mathbb{R}$ and the matrix
$B=\left\{  b_{ij}\right\}  _{i,j=1}^{N}$ satisfies the assumptions made by
Lanconelli-Polidoro in \cite{LP}, which make the corresponding operator with
constant $a_{ij}$ hypoelliptic. We construct an explicit fundamental solution
$\Gamma$ for $\mathcal{L}$, study its property, show a comparison result
between $\Gamma$ and the fundamental solution of some model operators with
constant $a_{ij}$, and show the unique solvability of the Cauchy problem for
$\mathcal{L}$ under various assumptions on the initial datum.

\end{abstract}

\section{Introduction}

We consider a Kolmogorov-Fokker-Planck (from now on KFP) operator of the kind:%
\begin{equation}
\mathcal{L}u=\sum_{i,j=1}^{q}a_{ij}\left(  t\right)  \partial_{x_{i}x_{j}}%
^{2}u+\sum_{k,j=1}^{N}b_{jk}x_{k}\partial_{x_{j}}u-\partial_{t}u,\qquad
(x,t)\in\mathbb{R}^{N+1} \label{L}%
\end{equation}
where:

(H1) $A_{0}\left(  t\right)  =\left\{  a_{ij}\left(  t\right)  \right\}
_{i,j=1}^{q}$ is a symmetric uniformly positive matrix on $\mathbb{R}^{q}$,
$q\leq N$, of bounded measurable coefficients defined for $t\in\mathbb{R}$, so
that
\begin{equation}
\nu\left\vert \xi\right\vert ^{2}\leq\sum_{i,j=1}^{q}a_{ij}\left(  t\right)
\xi_{i}\xi_{j}\leq\nu^{-1}\left\vert \xi\right\vert ^{2} \label{nu}%
\end{equation}
for some constant $\nu>0$, every $\xi\in\mathbb{R}^{q}$, a.e. $t\in\mathbb{R}$.

Lanconelli-Polidoro in \cite{LP} have studied the operators (\ref{L}) with
constant $a_{ij}$, proving that they are hypoelliptic if and only if the
matrix $B=\left\{  b_{ij}\right\}  _{i,j=1}^{N}$ satisfies the following
condition. There exists a basis of $\mathbb{R}^{N}$ such that $B$ assumes the
following form:

(H2) For $m_{0}=q$ and suitable positive integers $m_{1},\dots,m_{\kappa}$
such that
\begin{equation}
m_{0}\geq m_{1}\geq\ldots\geq m_{\kappa}\geq1,\quad\mathrm{and}\quad
m_{0}+m_{1}+\ldots+m_{\kappa}=N, \label{m-cond}%
\end{equation}
we have%
\begin{equation}
B=%
\begin{bmatrix}
\ast & \ast & \ldots & \ast & \ast\\
B_{1} & \ast & \ldots & \ast & \ast\\
\mathbb{O} & B_{2} & \ldots & \ast & \ast\\
\vdots & \vdots & \ddots & \vdots & \vdots\\
\mathbb{O} & \mathbb{O} & \ldots & B_{\kappa} & \ast
\end{bmatrix}
\label{B}%
\end{equation}
where every block $B_{j}$ is a $m_{j}\times m_{j-1}$ matrix of rank $m_{j}$
with $j=1,2,\ldots,\kappa$, while the entries of the blocks denoted by $\ast$
are arbitrary.

It is also proved in \cite{LP} that the operator $\mathcal{L}$ (corresponding
to constant $a_{ij}$) is left invariant with respect to a suitable
(noncommutative) Lie group of translations in $\mathbb{R}^{N}$. If, in
addition, all the blocks $\ast$ in (\ref{B}) vanish, then $\mathcal{L}$ is
also $2$-homogeneous with respect to a family of dilations. In this very
special case, the operator $\mathcal{L}$ fits into the rich theory of left
invariant, $2$-homogeneus, H\"{o}rmander operators on homoegeneous groups.

Coming back to the family of hypoelliptic and left invariant operators with
constant $a_{ij}$ (and possibly nonzero blocks $\ast$ in (\ref{B})), an
explicit fundamental solution is known, after \cite{K1} and \cite{LP}.

\bigskip

A first result of this paper consists in showing that if, under the same
structural assumptions considered in \cite{LP}, the coefficients $a_{ij}$ are
allowed to depend on $t$, even just in an $L^{\infty}$-way, then an explicit
fundamental solution $\Gamma$ can still be costructed. It is worth noting
that, under our assumptions (H1)-(H2), $\mathcal{L}$ is hypoelliptic if and
only if the coefficients $a_{i,j}$'s are $C^{\infty}$ functions, which also
means that $\Gamma$ is smooth outside the pole. In our more general context,
$\Gamma$ will be smooth in $x$ and only locally Lipschitz continuous in $t$,
outside the pole. Our fundamental solution also allows to solve a Cauchy
problem for $\mathcal{L}$ under various assumptions on the initial datum, and
to prove its uniqueness. Moreover, we show that the fundamental solution of
$\mathcal{L}$ satisfies two-sided bounds in terms of the fundamental solutions
of model operators of the kind:%
\begin{equation}
\mathcal{L}_{\alpha}u=\alpha\sum_{i=1}^{q}\partial_{x_{i}x_{i}}^{2}%
u+\sum_{k,j=1}^{N}b_{jk}x_{k}\partial_{x_{j}}u-\partial_{t}u, \label{L-alpha}%
\end{equation}
whose explicit expression is more easily handled. This fact has other
interesting consequences when combined with the results of \cite{LP}, which
allow to compare the fundamental solution of (\ref{L-alpha}) with that of the
corresponding \textquotedblleft principal part operator\textquotedblright,
which is obtained from (\ref{L-alpha}) by annihilating all the blocks $\ast$
in (\ref{B}). The fundamental solution of the latter operator has an even
simpler explicit form, since it possesses both translation invariance and homogeneity.

\bigskip

To put our results into context, let us now make some historical remarks.
Already in 1934, Kolmogorov in \cite{Ko} exhibited an explicit fundamental
solution, smooth outside the pole, for the ultraparabolic operator%
\[
\partial_{xx}^{2}+x\partial_{y}-\partial_{t}\text{ in }\mathbb{R}^{3}.
\]
For more general classes of ultraparabolic KFP operators, Weber \cite{We},
1951, Il'in \cite{il}, 1964, Sonin \cite{So}, 1967, proved the existence of a
fundamental solution smooth outside the pole, by the Levi method, starting
with an approximate fundamental solution which was inspired by the one found
by Kolmogorov. H\"{o}rmander, in the introduction of \cite{Hor2}, 1967,
sketches a procedure to compute explicitly (by Fourier transform and the
method of characteristics) a fundamental solution for a class of KFP operators
of type (\ref{L}) (with constant $a_{ij}$). In all the aforementioned papers
the focus is to prove that the operator, despite of its degenerate character,
is hypoelliptic. This is accomplished by showing the existence of a
fundamental solution smooth outside the pole, without explicitly computing it.

Kupcov in \cite{K1}, 1972, computes the fundamental solution for a class of
KFP operators of the kind (\ref{L}) (with constant $a_{ij}$). This procedure
is generalized by the same author in \cite{K2}, 1982, to a class of operators
(\ref{L}) with time-dependent coefficients $a_{ij}$, which however are assumed
of class $C^{\kappa}$ for some positive integer $\kappa$ related to the
structure of the matrix $B$. Our procedure to compute the fundamental solution
follows the technique by H\"{o}rmander (different from that of Kupcov) and
works also for nonsmooth $a_{ij}\left(  t\right)  $.

Based on the explicit expression of the fundamental solution, existence,
uniqueness and regularity issues for the Cauchy problem have been studied in
the framework of the semigroup setting. We refer here to the article by
Lunardi \cite{Lu}, and to Farkas and Lorenzi \cite{FL}. The parametrix method
introduced in \cite{We,il,So} was used by Polidoro in \cite{Po} and by Di
Francesco and Pascucci in \cite{DiFraPasc} for more general families of
Kolmogorov equations with H\"{o}lder continuous coefficients. We also refer to
the article \cite{DelarMen} by Delaure and Menozzi, where a Lipschitz
continuous drift term is considered in the framework of the stochastic theory.
For a recent survey on the theory of KFP operators we refer to the paper
\cite{AP} by Anceschi-Polidoro, while a discussion on several motivations to
study this class of operators can be found for instance in the survey book
\cite[\S 2.1]{BSur}.

\bigskip

The interest in studying KFP operators with a possibly rough time-depen\-dence
of the coefficients comes from the theory of stochastic processes. Indeed, let
$\sigma=\sigma(t)$ be a $N\times q$ matrix, with zero entries under the $q$-th
row, let $B$ as in (\ref{B}), and let $(W_{t})_{t\geq t_{0}}$ be a
$q$-dimensional Wiener process. Denote by $(X_{t})_{t\geq t_{0}}$ the solution
to the following $N$-dimensional stochastic differential equation
\begin{equation}%
\begin{cases}
dX_{t}=-BX_{t}\,dt+\sigma(t)\,dW_{t}\\
X_{t_{0}}=x_{0}.
\end{cases}
\label{stoc}%
\end{equation}
Then the \emph{forward Kolmogorov operator} $\mathcal{K}_{f}$ of
$(X_{t})_{t\geq t_{0}}$ agrees with $\mathcal{L}$ up to a constant zero order
term:
\[
\mathcal{K}_{f}v(x,t)=\mathcal{L}v(x,t)+\operatorname*{tr}(B)v(x,t),
\]
where
\begin{equation}
a_{ij}\left(  t\right)  =\tfrac{1}{2}\sum_{k=1}^{q}\sigma_{ik}(t)\sigma
_{jk}(t)\text{ \ \ }i,j=1,...,q. \label{a-def}%
\end{equation}
Moreover, the \emph{backward Kolmogorov operator} $\mathcal{K}_{b}$ of
$(X_{t})_{t\geq t_{0}}$ acts as follows
\[
\mathcal{K}_{b}u(y,s)=\partial_{s}u(y,s)+\sum\limits_{i,j=1}^{q}%
a_{ij}(s)\partial_{y_{i}y_{j}}^{2}u(y,s)-\sum\limits_{i,j=1}^{N}b_{ij}%
y_{j}\partial_{y_{i}}u(y,s).
\]
Note that $\mathcal{K}_{f}$ is the transposed operator of $\mathcal{K}_{b}$.
In general, given a differential operator $\mathcal{K}$, its transposed
operator $\mathcal{K}^{\ast}$ is the one which satisfies the relation%
\[
\int_{\mathbb{R}^{N+1}}\phi\left(  x,t\right)  \mathcal{K}^{\ast}\psi\left(
x,t\right)  dxdt=\int_{\mathbb{R}^{N+1}}\mathcal{K}\phi\left(  x,t\right)
\psi\left(  x,t\right)  dxdt
\]
for every $\phi,\psi\in C_{0}^{\infty}\left(  \mathbb{R}^{N+1}\right)  .$

A further motivation for our study is the following one. A regularity theory
for the operator $\mathcal{L}$ with H\"{o}lder continuous coefficients has
been developed by several authors (see e.g. \cite{Ma}, \cite{Lu},
\cite{DiFraPo}). However, as Pascucci and Pesce show in the Example 1.3 of
\cite{PaPe}, the requirement of H\"{o}lder continuity in $\left(  x,t\right)
$ with respect to the control distance may be very restrictive, due to the
interaction of time and space variable in the drift term of $\mathcal{L}$. In
view of this, a regularity requirement with respect to $x$-variables alone,
for $t$ fixed, with a possible rough dependence on $t$, seems a more natural
assumption. This paper can be seen as a first step to study KFP operators with
coefficients measurable in time and H\"{o}lder continuous or VMO in space, to
overcome the objection pointed out in \cite{PaPe}. For these operators the
fundamental solution of (\ref{L}) could be used as a parametrix, as done in
\cite{PaPe2}, to build a fundamental solution.

\begin{notation}
Throughout the paper we will regard vectors $x\in\mathbb{R}^{N}$ as columns,
and, we will write $x^{T},M^{T}$ to denote the transpose of a vector $x$ or a
matrix $M$. We also define the (symmetric, nonnegative) $N\times N$ matrix%
\begin{equation}
A\left(  t\right)  =%
\begin{bmatrix}
A_{0}\left(  t\right)  & \mathbb{O}\\
\mathbb{O} & \mathbb{O}%
\end{bmatrix}
. \label{A(t)}%
\end{equation}

\end{notation}

Before stating our results, let us fix precise definitions of solution to the
equation $\mathcal{L}u=0$ and to a Cauchy problem for $\mathcal{L}$.

\begin{definition}
\label{Def solution}We say that $u\left(  x,t\right)  $ is a solution to the
equation $\mathcal{L}u=0$ in $\mathbb{R}^{N}\times I$, for some open interval
$I$, if:

$u$ is jointly continuous in $\mathbb{R}^{N}\times I$;

for every $t\in I,$ $u\left(  \cdot,t\right)  \in C^{2}\left(  \mathbb{R}%
^{N}\right)  $;

for every $x\in\mathbb{R}^{N}$, $u\left(  x,\cdot\right)  $ is absolutely
continuous on $I$, and $\frac{\partial u}{\partial t}$ (defined for a.e. $t$)
is essentially bounded for $t$ ranging in every compact subinterval of $I$;

for a.e. $t\in I$ and every $x\in\mathbb{R}^{N}$, $\mathcal{L}u\left(
x,t\right)  =0$.
\end{definition}

\begin{definition}
\label{Def Cauchy}We say that $u\left(  x,t\right)  $ is a solution to the
Cauchy problem%
\begin{equation}
\left\{
\begin{array}
[c]{l}%
\mathcal{L}u=0\text{ in }\mathbb{R}^{N}\times\left(  t_{0},T\right) \\
u\left(  \cdot,t_{0}\right)  =f
\end{array}
\right.  \label{PdC}%
\end{equation}
for some $T\in(-\infty,+\infty]$, $t_{0}\in\left(  -\infty,T\right)  $, where
$f$ is continuous in $\mathbb{R}^{N}$ or belongs to $L^{p}\left(
\mathbb{R}^{N}\right)  $ for some $p\in\lbrack1,\infty)$ if:

(a) $u$ is a solution to the equation $\mathcal{L}u=0$ in $\mathbb{R}%
^{N}\times\left(  t_{0},T\right)  $ (in the sense of the above definition);

(b$_{1}$) if $f\in C^{0}\left(  \mathbb{R}^{N}\right)  $ then $u\left(
x,t\right)  \rightarrow f\left(  x_{0}\right)  $ as $\left(  x,t\right)
\rightarrow\left(  x_{0},t_{0}^{+}\right)  $,$\ $for every $x_{0}\in
\mathbb{R}^{N}$;

(b$_{2}$) if $f\in L^{p}\left(  \mathbb{R}^{N}\right)  $ for some $p\in
\lbrack1,\infty)$ then $u\left(  \cdot,t\right)  \in L^{p}\left(
\mathbb{R}^{N}\right)  $ for every $t\in\left(  t_{0},T\right)  $, and
$\left\Vert u\left(  \cdot,t\right)  -f\right\Vert _{L^{p}\left(
\mathbb{R}^{N}\right)  }\rightarrow0$ as $t\rightarrow t_{0}^{+}.$
\end{definition}

In the following, we will also need the \emph{transposed operator }of
$\mathcal{L}$, defined by
\begin{equation}
\mathcal{L}^{\ast}u=\sum_{i,j=1}^{q}a_{ij}\left(  s\right)  \partial
_{y_{i}y_{j}}^{2}u-\sum_{k,j=1}^{N}b_{jk}y_{k}\partial_{y_{j}}%
u-u\operatorname*{Tr}B+\partial_{s}u. \label{L star}%
\end{equation}
The definition of solution to the equation $\mathcal{L}^{\ast}u=0$ is
perfectly analogous to Definition \ref{Def solution}.

We can now state precisely the main results of the paper.

\begin{theorem}
\label{Thm main}Under the assumptions (H1)-(H2) above, denote by $E(s)$ and
$C(t,t_{0})$ the following $N\times N$ matrices
\begin{equation}
E\left(  s\right)  =\exp\left(  -sB\right)  ,\qquad C\left(  t,t_{0}\right)
=\int_{t_{0}}^{t}E\left(  t-\sigma\right)  A\left(  \sigma\right)  E\left(
t-\sigma\right)  ^{T}d\sigma\label{eq-EC}%
\end{equation}
for $s,t,t_{0}\in\mathbb{R}$ and $t>t_{0}$. Then the matrix $C\left(
t,t_{0}\right)  $ is symmetric and positive for every $t>t_{0}$. Let%
\begin{align}
&  \Gamma\left(  x,t;x_{0},t_{0}\right) \nonumber\\
&  =\frac{1}{\left(  4\pi\right)  ^{N/2}\sqrt{\det C\left(  t,t_{0}\right)  }%
}e^{-\left(  \frac{1}{4}\left(  x-E\left(  t-t_{0}\right)  x_{0}\right)
^{T}C\left(  t,t_{0}\right)  ^{-1}\left(  x-E\left(  t-t_{0}\right)
x_{0}\right)  +\left(  t-t_{0}\right)  \operatorname*{Tr}B\right)  }
\label{Gamma}%
\end{align}
for $t>t_{0}$, $\Gamma=0$ for $t\leq t_{0}$. Then $\Gamma$ has the following
properties (so that $\Gamma$ is a fundamental solution for $\mathcal{L}$ with
pole $\left(  x_{0},t_{0}\right)  $).

(i) In the region%
\begin{equation}
\mathbb{R}_{\ast}^{2N+2}=\left\{  \left(  x,t,x_{0},t_{0}\right)
\in\mathbb{R}^{2N+2}:\left(  x,t\right)  \neq\left(  x_{0},t_{0}\right)
\right\}  \label{R star}%
\end{equation}
the function $\Gamma$ is jointly continuous in $\left(  x,t,x_{0}%
,t_{0}\right)  $ and is $C^{\infty}$ with respect to $x,x_{0}$. The functions
$\frac{\partial^{\alpha+\beta}\Gamma}{\partial x^{\alpha}\partial x_{0}%
^{\beta}}$ (for every multiindices $\alpha,\beta$) are jointly continuous in
$\left(  x,t,x_{0},t_{0}\right)  \in\mathbb{R}_{\ast}^{2N+2}$. Moreover
$\Gamma$ and and $\frac{\partial^{\alpha+\beta}\Gamma}{\partial x^{\alpha
}\partial x_{0}^{\beta}}$ are Lipschitz continuous with respect to $t$ and
with respect to $t_{0}$ in any region $H\leq t_{0}+\delta\leq t\leq K$ for
fixed $H,K\in\mathbb{R}$ and $\delta>0$.

$\lim_{\left\vert x\right\vert \rightarrow+\infty}\Gamma\left(  x,t;x_{0}%
,t_{0}\right)  =0$ for every $t>t_{0}$ and every $x_{0}\in\mathbb{R}^{N}$.

$\lim_{\left\vert x_{0}\right\vert \rightarrow+\infty}\Gamma\left(
x,t;x_{0},t_{0}\right)  =0$ for every $t>t_{0}$ and every $x\in\mathbb{R}^{N}$.

(ii) For every fixed $\left(  x_{0},t_{0}\right)  \in\mathbb{R}^{N+1}$, the
function $\Gamma\left(  \cdot,\cdot;x_{0},t_{0}\right)  $ is a solution to
$\mathcal{L}u=0$ in $\mathbb{R}^{N}\times\left(  t_{0},+\infty\right)  $ (in
the sense of Definition \ref{Def solution});

(iii) For every fixed $\left(  x,t\right)  \in\mathbb{R}^{N+1}$, the function
$\Gamma\left(  x,t;\cdot,\cdot\right)  $ is a solution to $\mathcal{L}^{\ast
}u=0$ in $\mathbb{R}^{N}\times\left(  -\infty,t\right)  $;

(iv) Let $f\in C_{b}^{0}\left(  \mathbb{R}^{N}\right)  $ (bounded continuous),
or $f\in L^{p}\left(  \mathbb{R}^{N}\right)  $ for some $p\in\lbrack1,\infty
)$.\ Then there exists one and only one solution to the Cauchy problem
(\ref{PdC}) (in the sense of Definition \ref{Def Cauchy}, with $T=\infty$)
such that $u\in C_{b}^{0}\left(  \mathbb{R}^{N}\times\lbrack t_{0}%
,\infty\right)  $ or $u\left(  t,\cdot\right)  \in L^{p}\left(  \mathbb{R}%
^{N}\right)  $ for every $t>t_{0}$, respectively. The solution is given by%
\begin{equation}
u\left(  x,t\right)  =\int_{\mathbb{R}^{N}}\Gamma\left(  x,t;y,t_{0}\right)
f\left(  y\right)  dy \label{repr formula}%
\end{equation}
and is $C^{\infty}\left(  \mathbb{R}^{N}\right)  $ with respect to $x$ for
every fixed $t>t_{0}$. If moreover $f$ is continuous and vanishes at infinity,
then $u\left(  \cdot,t\right)  \rightarrow f$ uniformly in $\mathbb{R}^{N}$ as
$t\rightarrow t_{0}^{+}$.

(v) Let $f$ be a (possibly unbounded) continuous function on $\mathbb{R}^{N}$
satisfying the condition
\begin{equation}
\int_{\mathbb{R}^{N}}\left\vert f\left(  x\right)  \right\vert e^{-\alpha
\left\vert x\right\vert ^{2}}dx<\infty, \label{exp bound}%
\end{equation}
for some $\alpha>0$. Then there exists $T>0$ such that there exists one and
only one solution $u$ to the Cauchy problem (\ref{PdC}) satisfying condition
\begin{equation}
\int\limits_{t_{0}}^{T}\int\limits_{\mathbb{R}^{N}}\left\vert
u(x,t)\right\vert \,e^{-C|x|^{2}}\,dx\,dt<+\infty\label{cond uniqueness}%
\end{equation}
for some $C>0$. The solution $u\left(  x,t\right)  $ is given by
(\ref{repr formula}) for $t\in(t_{0},T)$. It is $C^{\infty}\left(
\mathbb{R}^{N}\right)  $ with respect to $x$ for every fixed $t\in\left(
t_{0},T\right)  $.

(vi) $\Gamma$ satisfies for every $x_{0}\in\mathbb{R}^{N}$, $t_{0}<t$ the
integral identities
\begin{align*}
\int_{\mathbb{R}^{N}}\Gamma\left(  x_{0},t;y,t_{0}\right)  dy  &  =1\\
\int_{\mathbb{R}^{N}}\Gamma\left(  x,t;x_{0},t_{0}\right)  dx  &  =e^{-\left(
t-t_{0}\right)  \operatorname*{Tr}B}.
\end{align*}

(vii) $\Gamma$ satisfies the reproduction formula%
\[
\Gamma\left(  x,t;y,s\right)  =\int_{\mathbb{R}^{N}}\Gamma\left(
x,t;z,\tau\right)  \Gamma\left(  z,\tau;y,s\right)  dz
\]
for every $x,y\in\mathbb{R}^{N}$ and $s<\tau<t$.
\end{theorem}

\begin{remark}
Our uniqueness results only require the condition (\ref{cond uniqueness}).
Indeed, as we will prove in Proposition \ref{Prop cond uniqueness} all the
solutions to the Cauchy problem (\ref{PdC}), in the sense of Definition
\ref{Def Cauchy}, with $f\in L^{p}\left(  \mathbb{R}^{N}\right)  $ for some
$p\in\lbrack1,\infty)$, $f$ $\in C_{b}^{0}\left(  \mathbb{R}^{N}\right)  $ or
$f\in C^{0}\left(  \mathbb{R}^{N}\right)  $ with $f$ satisfying
(\ref{exp bound}), do satisfy the condition (\ref{cond uniqueness}).
\end{remark}

\begin{remark}
All the statements in the above theorem still hold if the coefficients
$a_{ij}\left(  t\right)  $ are defined only for $t$ belonging to some interval
$I$. In this case the above formulas need to be considered only for
$t,t_{0}\in I$. In order to simplify notation, throughout the paper we will
only consider the case $I=\mathbb{R}$.
\end{remark}

The above theorem will be proved in section \ref{sec properties gamma}.

The second main result of this paper is a comparison between $\Gamma$ and the
fundamental solutions $\Gamma_{\alpha}$ of the model operators (\ref{L-alpha})
corresponding to $\alpha=\nu,\alpha=\nu^{-1}$ (with $\nu$ as in (\ref{nu})).
Specializing (\ref{Gamma}) to the operators (\ref{L-alpha}) we have%
\[
\Gamma_{\alpha}\left(  x,t;x_{0},t_{0}\right)  =\Gamma_{\alpha}\left(
x-E\left(  t-t_{0}\right)  x_{0},t-t_{0};0,0\right)
\]
with%
\begin{equation}
\Gamma_{\alpha}\left(  x,t;0,0\right)  =\frac{1}{\left(  4\pi\alpha\right)
^{N/2}\sqrt{\det C_{0}\left(  t\right)  }}e^{-\left(  \frac{1}{4\alpha}%
x^{T}C_{0}\left(  t\right)  ^{-1}x+t\operatorname*{Tr}B\right)  }
\label{G_nu polo0}%
\end{equation}
where, here and in the following, $C_{0}\left(  t\right)  =C\left(
t,0\right)  $ with $A_{0}\left(  t\right)  =I_{q}$ (identity $q\times q$
matrix). Explicitly:%
\begin{equation}
C_{0}\left(  t\right)  =\int_{0}^{t}E\left(  t-\sigma\right)  I_{q,N}E\left(
t-\sigma\right)  ^{T}d\sigma\text{,} \label{C_0}%
\end{equation}
where $I_{q,N}$ is the $N\times N$ matrix given by%
\[
I_{q,N}=%
\begin{bmatrix}
I_{q} & 0\\
0 & 0
\end{bmatrix}
.
\]

Then:

\begin{theorem}
\label{Thm comparison Gammas}For every $t>t_{0}$ and $x,x_{0}\in\mathbb{R}%
^{N}$ we have%
\begin{equation}
\nu^{N}\Gamma_{\nu}\left(  x,t;x_{0},t_{0}\right)  \leq\Gamma\left(
x,t;x_{0},t_{0}\right)  \leq\frac{1}{\nu^{N}}\Gamma_{\nu^{-1}}\left(
x,t;x_{0},t_{0}\right)  . \label{G G_nu}%
\end{equation}
\ 
\end{theorem}

The above theorem will be proved in section \ref{sec comparison}. The
following example illustrates the reason why our comparison result is useful.

\begin{example}
Let us consider the operator%
\[
\mathcal{L}u=a\left(  t\right)  u_{x_{1}x_{1}}+x_{1}u_{x_{2}}-u_{t}%
\]
with $a\left(  t\right)  $ measurable and satisfying%
\[
0<\nu\leq a\left(  t\right)  \leq\nu^{-1}\text{ for every }t\in\mathbb{R}.
\]
Let us compute $\Gamma\left(  x,t;0,0\right)  $ in this case. We have:%
\begin{align*}
A  &  =%
\begin{bmatrix}
a\left(  t\right)  & 0\\
0 & 0
\end{bmatrix}
;B=%
\begin{bmatrix}
0 & 0\\
1 & 0
\end{bmatrix}
;E\left(  s\right)  =%
\begin{bmatrix}
1 & 0\\
-s & 1
\end{bmatrix}
;\\
C\left(  t\right)   &  \equiv C\left(  t,0\right)  =\int_{0}^{t}%
\begin{bmatrix}
1 & 0\\
-s & 1
\end{bmatrix}%
\begin{bmatrix}
a\left(  t-s\right)  & 0\\
0 & 0
\end{bmatrix}%
\begin{bmatrix}
1 & -s\\
0 & 1
\end{bmatrix}
ds\\
&  =\int_{0}^{t}a\left(  t-s\right)
\begin{bmatrix}
1 & -s\\
-s & s^{2}%
\end{bmatrix}
ds\\
&  \left(  \text{after two integrations by parts}\right) \\
&  =%
\begin{bmatrix}
a^{\ast}\left(  t\right)  & -a^{\ast\ast}\left(  t\right) \\
-a^{\ast\ast}\left(  t\right)  & 2a^{\ast\ast\ast}\left(  t\right)
\end{bmatrix}
\end{align*}
where we have set:%
\[
a^{\ast}\left(  t\right)  =\int_{0}^{t}a\left(  s\right)  ds;a^{\ast\ast
}\left(  t\right)  =\int_{0}^{t}a^{\ast}\left(  s\right)  ds;a^{\ast\ast\ast
}\left(  t\right)  =\int_{0}^{t}a^{\ast\ast}\left(  s\right)  ds.
\]
Therefore we find, for $t>0$:%
\[
\Gamma\left(  x,t;0,0\right)  =\frac{1}{4\pi\sqrt{\det C\left(  t\right)  }%
}e^{-\left(  \frac{1}{4}x^{T}C\left(  t\right)  ^{-1}x\right)  }%
\]
with%
\[
C\left(  t\right)  ^{-1}=\frac{1}{\det C\left(  t\right)  }%
\begin{bmatrix}
2a^{\ast\ast\ast}\left(  t\right)  & a^{\ast\ast}\left(  t\right) \\
a^{\ast\ast}\left(  t\right)  & a^{\ast}\left(  t\right)
\end{bmatrix}
\]
so that, explicitly, we have%
\begin{align*}
&  \Gamma\left(  x,t;0,0\right) \\
&  =\frac{1}{4\pi\sqrt{\det C\left(  t\right)  }}\exp\left(  -\frac{\left(
2a^{\ast\ast\ast}\left(  t\right)  x_{1}^{2}+2a^{\ast\ast}\left(  t\right)
x_{1}x_{2}+a^{\ast}\left(  t\right)  x_{2}^{2}\right)  }{4\det C\left(
t\right)  }\right) \\
&  \text{with }\det C\left(  t\right)  =2a^{\ast}\left(  t\right)  a^{\ast
\ast\ast}\left(  t\right)  -a^{\ast\ast}\left(  t\right)  ^{2}.
\end{align*}

On the other hand, when considering the model operator
\[
L_{\alpha}u=\alpha u_{x_{1}x_{1}}+x_{1}u_{x_{2}}-u_{t}%
\]
with constant $\alpha>0$, we have%
\[
\Gamma_{\alpha}\left(  x,t;0,0\right)  =\frac{\sqrt{3}}{2\pi\alpha t^{2}}%
\exp\left(  -\frac{1}{\alpha}\left(  \frac{x_{1}^{2}}{t}+\frac{3x_{1}x_{2}%
}{t^{2}}+\frac{3x_{2}^{2}}{t^{3}}\right)  \right)  .
\]
The comparison result of Theorem \ref{Thm comparison Gammas} then reads as
follows:%
\[
\nu^{2}\Gamma_{\nu}\left(  x,t;0,0\right)  \leq\Gamma\left(  x,t;0,0\right)
\leq\frac{1}{\nu^{2}}\Gamma_{\nu^{-1}}\left(  x,t;0,0\right)
\]
or, explicitly,%
\begin{align*}
&  \nu\frac{\sqrt{3}}{2\pi t^{2}}\exp\left(  -\frac{1}{\nu}\left(  \frac
{x_{1}^{2}}{t}+\frac{3x_{1}x_{2}}{t^{2}}+\frac{3x_{2}^{2}}{t^{3}}\right)
\right)  \leq\Gamma\left(  x,t;0,0\right) \\
&  \leq\frac{1}{\nu}\frac{\sqrt{3}}{2\pi t^{2}}\exp\left(  -\nu\left(
\frac{x_{1}^{2}}{t}+\frac{3x_{1}x_{2}}{t^{2}}+\frac{3x_{2}^{2}}{t^{3}}\right)
\right)  .
\end{align*}

\end{example}

\bigskip

\textbf{Plan of the paper}. In \S \ref{sec computation Gamma} we compute the
explicit expression of the fundamental solution $\Gamma$ of $\mathcal{L}$ by
using the Fourier transform and the method of characteristics, showing how one
arrives to the the explicit formula (\ref{Gamma}). This procedure is somehow
formal as, due to the nonsmoothness of the coefficients $a_{ij}\left(
t\right)  $, we cannot plainly assume that the functional setting where the
construction is done is the usual distributional one. Since all the properties
of $\Gamma$ which qualify it as a fundamental solution will be proved in the
subsequent sections, on a purely logical basis one could say that
\S \ref{sec computation Gamma} is superfluous. Nevertheless, we prefer to
present this complete computation to show how this formula has been built. A
further reason to do this is the following one. The unique article where the
analogous computation in the constant coefficient case is written in detail
seems to be \cite{K1}, and it is written in Russian language.

In \S \ref{sec comparison} we prove Theorem \ref{Thm comparison Gammas},
comparing $\Gamma$ with the fundamental solutions of two model operators,
which is easier to write explicitly and to study. In
\S \ref{sec properties gamma} we will prove Theorem \ref{Thm main}, namely:
point (i) in \S \ref{sec regularity}; points (ii), (iii), (vi) in
\S \ref{sec solves equation}; points (iv), (v), (vii) in \S \ref{sec Cauchy}.

\section{Computation of the fundamental solution $\Gamma$%
\label{sec computation Gamma}}

As explained at the end of the introduction, this section contains a formal
computation of the fundamental solution $\Gamma$. To this aim, we choose any
$(x_{0},t_{0})\in\mathbb{R}^{N+1}$, and we look for a solution to the Cauchy
Problem
\begin{equation}
\left\{
\begin{tabular}
[c]{ll}%
$\mathcal{L}u=0$ & $\text{for }x\in\mathbb{R}^{N},t>t_{0}$\\
$u\left(  \cdot,t_{0}\right)  =\delta_{x_{0}}$ & in $\mathcal{D}^{\prime
}\left(  \mathbb{R}^{N}\right)  $%
\end{tabular}
\right.  \label{eq-Cauchy-Gamma}%
\end{equation}
by applying the Fourier transform with respect to $x$, and using the notation
\[
\widehat{u}\left(  \xi,t\right)  =\mathcal{F}\left(  u\left(  \cdot,t\right)
\right)  \left(  \xi\right)  :=\int_{\mathbb{R}^{N}}e^{-2\pi ix^{T}\xi
}u(x,t)dx.
\]
We have:%
\[
\sum_{i,j=1}^{q}a_{ij}\left(  t\right)  \left(  -4\pi^{2}\xi_{i}\xi
_{j}\right)  \widehat{u}+\sum_{k,j=1}^{N}b_{jk}\mathcal{F}\left(
x_{k}\partial_{x_{j}}u\right)  -\partial_{t}\widehat{u}=0.
\]
By the standard properties of the Fourier transform, it follows that
\[
\mathcal{F}\left(  x_{k}\partial_{x_{j}}u\right)  =\frac{1}{-2\pi i}%
\partial_{\xi_{k}}\left(  \mathcal{F}\left(  \partial_{x_{j}}u\right)
\right)  =\frac{1}{-2\pi i}\partial_{\xi_{k}}\left(  2\pi i\xi_{j}%
\widehat{u}\right)  =-\left(  \delta_{jk}\widehat{u}+\xi_{j}\partial_{\xi_{k}%
}\widehat{u}\right)  .
\]
then the problem (\ref{eq-Cauchy-Gamma}) is equivalent to the following Cauchy
problem that we write in compact form (recalling the definition of the
$A\left(  t\right)  \ $given in (\ref{A(t)})) as%
\begin{equation}
\left\{
\begin{array}
[c]{l}%
\left(  \nabla_{\xi}\widehat{u}(\xi,t)\right)  ^{T}B^{T}\xi+\partial
_{t}\widehat{u}(\xi,t)=-\left(  4\pi^{2}\xi^{T}A\left(  t\right)
\xi+\operatorname*{Tr}B\right)  \widehat{u}(\xi,t),\\
\\
\widehat{u}\left(  \xi,t_{0}\right)  =e^{-2\pi i\xi^{T}x_{0}}.
\end{array}
\right.  \label{eq-Cauchy-Gamma-hat}%
\end{equation}
Now we solve the problem (\ref{eq-Cauchy-Gamma-hat}) by the method of
characteristics. Fix any initial condition $\eta\in\mathbb{R}^{N}$, and
consider the system of ODEs:%
\begin{equation}
\left\{
\begin{array}
[c]{ll}%
\frac{d\xi}{ds}\left(  s\right)  =B^{T}\xi(s), & \xi\left(  0\right)  =\eta,\\
& \\
\frac{dt}{ds}\left(  s\right)  =1, & t\left(  0\right)  =t_{0},\\
& \\
\frac{dz}{ds}\left(  s\right)  =-\left(  4\pi^{2}\xi^{T}(s)A\left(
t(s)\right)  \xi(s)+\operatorname*{Tr}B\right)  z(s), & z\left(  0\right)
=e^{-2\pi i\eta^{T}x_{0}}.
\end{array}
\right.  \label{ODE}%
\end{equation}
We plainly find $t(s)=t_{0}+s$ and $\xi(s)=\exp\left(  sB^{T}\right)  \eta$,
so that the last equation becomes
\[
\frac{dz}{ds}\left(  s\right)  =-\left(  4\pi^{2}\left(  \exp\left(
sB^{T}\right)  \eta\right)  ^{T}A\left(  t_{0}+s\right)  \exp\left(
sB^{T}\right)  \eta+\operatorname*{Tr}B\right)  z\left(  s\right)  ,
\]
whose solution, with initial condition $z\left(  0\right)  =e^{-2\pi i\eta
^{T}x_{0}}$, is
\[
z\left(  s\right)  =\exp\left(  -4\pi^{2}\int_{0}^{s}\eta^{T}\left[
\exp\left(  \sigma B\right)  A\left(  t_{0}+\sigma\right)  \exp\left(  \sigma
B^{T}\right)  \right]  \eta d\sigma-s\operatorname*{Tr}B-2\pi i\eta^{T}%
x_{0}\right)  .
\]
Hence, substituting $s=t-t_{0},\eta=\exp\left(  \left(  t_{0}-t\right)
B^{T}\right)  \xi$, recalling the notation introduced in (\ref{eq-EC}), we
find%
\begin{align*}
&  \widehat{u}\left(  \xi,t\right)  =z(t-t_{0})\\
&  =\exp\left(  -4\pi^{2}\int_{0}^{t-t_{0}}\xi^{T}\exp\left(  \left(
t_{0}-t+\sigma\right)  B\right)  A\left(  t_{0}+\sigma\right)  \exp\left(
\left(  t_{0}-t+\sigma\right)  B^{T}\right)  \xi d\sigma\right. \\
&  \left.  -\left(  t-t_{0}\right)  \operatorname*{Tr}B-2\pi i\xi^{T}%
\exp\left(  \left(  t_{0}-t\right)  B\right)  x_{0}\frac{{}}{{}}\right)
\end{align*}%
\begin{align}
&  =\exp\left(  -4\pi^{2}\xi^{T}\left(  \int_{t_{0}}^{t}E\left(
\sigma-t\right)  A\left(  \sigma\right)  E\left(  \sigma-t\right)  ^{T}%
d\sigma\right)  \xi\right. \nonumber\\
&  \qquad\qquad\left.  -(t-t_{0})\operatorname*{Tr}B-2\pi i\xi^{T}E\left(
t-t_{0}\right)  x_{0}\frac{{}}{{}}\right) \nonumber\\
&  =\exp\left(  -4\pi^{2}\xi^{T}C\left(  t,t_{0}\right)  \xi-(t-t_{0}%
)\operatorname*{Tr}B-2\pi i\xi^{T}E\left(  t-t_{0}\right)  x_{0}\right)  .
\label{trasf Gamma zero}%
\end{align}

Let
\begin{align}
G\left(  \xi,t;x_{0},t_{0}\right)   &  =\exp\left(  -4\pi^{2}\xi^{T}C\left(
t,t_{0}\right)  \xi-\left(  t-t_{0}\right)  \operatorname*{Tr}B-2\pi i\xi
^{T}E\left(  t-t_{0}\right)  x_{0}\right) \nonumber\\
G_{0}\left(  \xi,t,t_{0}\right)   &  =\exp\left(  -4\pi^{2}\xi^{T}C\left(
t,t_{0}\right)  \xi\right)  \label{G G0}%
\end{align}
and note that if%
\[
\mathcal{F}\left(  k\left(  \cdot,t,t_{0}\right)  \right)  \left(  \xi\right)
=G_{0}\left(  \xi,t,t_{0}\right)
\]
then%
\begin{equation}
\mathcal{F}\left(  k\left(  \cdot-E\left(  t-t_{0}\right)  x_{0}%
,t,t_{0}\right)  \exp\left(  -\left(  t-t_{0}\right)  \operatorname*{Tr}%
B\right)  \right)  \left(  \xi\right)  =G\left(  \xi,t;x_{0},t_{0}\right)  ,
\label{F G G0}%
\end{equation}
hence it is enough to compute the antitransform of $G_{0}\left(  \xi
,t,t_{0}\right)  $. In order to do that, the following will be useful:

\begin{proposition}
\label{Prop Gauss transform}Let $A$ be an $N\times N$ real symmetric positive
constant matrix. Then:%
\[
\mathcal{F}\left(  e^{-\left(  x^{T}Ax\right)  }\right)  \left(  \xi\right)
=\left(  \frac{\pi^{N}}{\det A}\right)  ^{1/2}e^{-\pi^{2}\xi^{T}A^{-1}\xi}.
\]

\end{proposition}

The above formula is a standard known result in probability theory, being the
characteristic function of a multivariate normal distribution (see for
instance \cite[Prop. 1.1.2]{DP}).

To apply the previous proposition, and antitransform the function
$G_{0}\left(  \xi,t,t_{0}\right)  $, we still need to know that the matrix
$C\left(  t,t_{0}\right)  $ is strictly positive. By \cite{LP} we know that
the matrix $C_{0}\left(  t\right)  $ (see (\ref{C_0})) is positive, under the
structure conditions on $B$ expressed in (\ref{B}). Exploiting this fact, let
us show that the same is true for our $C\left(  t,t_{0}\right)  $:

\begin{proposition}
\label{Prop fq C C0}For every $\xi\in\mathbb{R}^{N}$ and every $t>t_{0}$ we
have%
\begin{equation}
\nu^{-1}\xi^{T}C_{0}\left(  t-t_{0}\right)  \xi\geq\xi^{T}C\left(
t,t_{0}\right)  \xi\geq\nu\xi^{T}C_{0}\left(  t-t_{0}\right)  \xi.
\label{ineq C C0}%
\end{equation}
In particular, the matrix $C\left(  t,t_{0}\right)  $ is positive for
$t>t_{0}$.
\end{proposition}

\begin{proof}%
\[
\xi^{T}C\left(  t,t_{0}\right)  \xi=\int_{t_{0}}^{t}\xi^{T}E\left(
t-s\right)  A\left(  s\right)  E\left(  t-s\right)  ^{T}\xi ds.
\]
Next, letting $E\left(  s\right)  =\left(  e_{ij}\left(  s\right)  \right)
_{i,j=1}^{N}$ and $\eta_{h}\left(  s\right)  =\sum_{k=1}^{N}\xi_{k}%
e_{kh}\left(  s\right)  $ we have%
\begin{align*}
\xi^{T}E\left(  t-s\right)  A\left(  s\right)  E\left(  t-s\right)  ^{T}\xi &
=\sum_{i,j,h,k=1}^{N}\xi_{i}e_{ij}\left(  t-s\right)  a_{jh}\left(  s\right)
e_{kh}\left(  t-s\right)  \xi_{k}\\
&  =\sum_{j,h=1}^{q}a_{jh}\left(  s\right)  \eta_{j}\left(  t-s\right)
\eta_{h}\left(  t-s\right)  \geq\nu\sum_{j=1}^{q}\eta_{j}\left(  t-s\right)
^{2}\\
&  =\nu\xi^{T}E\left(  t-s\right)  I_{q,N}E\left(  t-s\right)  ^{T}\xi
\end{align*}
where%
\[
I_{q,N}=%
\begin{bmatrix}
I_{q} & 0\\
0 & 0
\end{bmatrix}
.
\]
Integrating for $s\in\left(  t_{0},t\right)  $ the previous inequality we get%
\[
\xi^{T}C\left(  t,t_{0}\right)  \xi\geq\nu\xi^{T}\int_{t_{0}}^{t}E\left(
t-s\right)  I_{q,N}E\left(  t-s\right)  ^{T}ds\xi=\nu\xi^{T}C_{0}\left(
t-t_{0}\right)  \xi.
\]
Analogously we get the other bound.
\end{proof}

By the previous proposition, the matrix $C\left(  t,t_{0}\right)  $ is
positive definite for every $t>t_{0}$, since, under our assumptions, this is
true for $C_{0}\left(  t-t_{0}\right)  $. Therefore we can invert $C\left(
t,t_{0}\right)  $ and antitransform the function $G_{0}\left(  \xi
,t,t_{0}\right)  $ in (\ref{G G0}). Namely, applying Proposition
\ref{Prop Gauss transform} to $C\left(  t,t_{0}\right)  ^{-1}$ we get:%

\begin{align*}
&  \mathcal{F}\left(  e^{-\left(  x^{T}C\left(  t,t_{0}\right)  ^{-1}x\right)
}\right)  \left(  \xi\right)  =\pi^{N/2}\sqrt{\det C\left(  t,t_{0}\right)
}e^{-\pi^{2}\xi^{T}C\left(  t,t_{0}\right)  \xi}\\
&  \mathcal{F}\left(  \frac{1}{\left(  4\pi\right)  ^{N/2}\sqrt{\det C\left(
t,t_{0}\right)  }}e^{-\left(  \frac{1}{4}x^{TC\left(  t,t_{0}\right)
-1}x\right)  }\right)  \left(  \xi\right)  =e^{-4\pi^{2}\xi^{T}C\left(
t,t_{0}\right)  \xi}.
\end{align*}
Hence we have computed the antitransform of $G_{0}\left(  \xi,t,t_{0}\right)
$, and by (\ref{F G G0}) this also implies%
\begin{align*}
&  \mathcal{F}\left(  \frac{1}{\left(  4\pi\right)  ^{N/2}\sqrt{\det C\left(
t,t_{0}\right)  }}e^{-\left(  \frac{1}{4}\left(  x-E\left(  t-t_{0}\right)
x_{0}\right)  ^{T}C\left(  t,t_{0}\right)  ^{-1}\left(  x-E\left(
t-t_{0}\right)  x_{0}\right)  +\left(  t-t_{0}\right)  \operatorname*{Tr}%
B\right)  }\right)  \left(  \xi\right) \\
&  =\exp\left(  -4\pi^{2}\xi^{T}C\left(  t,t_{0}\right)  \xi-\left(
t-t_{0}\right)  \operatorname*{Tr}B-2\pi i\xi^{T}E\left(  t-t_{0}\right)
x_{0}\right)  .
\end{align*}
Hence the (so far, \textquotedblleft formal\textquotedblright) fundamental
solution of $\mathcal{L}$ is%
\begin{align*}
&  \Gamma\left(  x,t;x_{0},t_{0}\right) \\
&  =\frac{1}{\left(  4\pi\right)  ^{N/2}\sqrt{\det C\left(  t,t_{0}\right)  }%
}e^{-\left(  \frac{1}{4}\left(  x-E\left(  t-t_{0}\right)  x_{0}\right)
^{T}C\left(  t,t_{0}\right)  ^{-1}\left(  x-E\left(  t-t_{0}\right)
x_{0}\right)  +\left(  t-t_{0}\right)  \operatorname*{Tr}B\right)  },
\end{align*}
which is the expression given in Theorem \ref{Thm main}.

\section{Comparison between $\Gamma$ and fundamental solutions of model
operators\label{sec comparison}}

In this section we will prove Theorem \ref{Thm comparison Gammas}. The first
step is to derive from Proposition \ref{Prop fq C C0} an analogous control
between the quadratic forms associated to the inverse matrices $C_{0}\left(
t-t_{0}\right)  ^{-1},C\left(  t,t_{0}\right)  ^{-1}$. The following algebraic
fact will help:

\begin{proposition}
\label{Prop Pol LM}Let $C_{1},C_{2}$ be two real symmetric positive $N\times
N$ matrices. If%
\begin{equation}
\xi^{T}C_{1}\xi\leq\xi^{T}C_{2}\xi\text{ for every }\xi\in\mathbb{R}^{N}
\label{C1<C2}%
\end{equation}
then%
\[
\xi^{T}C_{2}^{-1}\xi\leq\xi^{T}C_{1}^{-1}\xi\text{ for every }\xi\in
\mathbb{R}^{N}%
\]
and%
\[
\det C_{1}\leq\det C_{2}.
\]

\end{proposition}

The first implication is already proved in \cite[Remark 2.1.]{Po}. For
convenience of the reader, we write a proof of both.

\bigskip

\begin{proof}
Let us fix some shorthand notation. Whenever (\ref{C1<C2}) holds for two
symmetric positive matrices, we will write $C_{1}\leq C_{2}$. Note that for
every symmetric $N\times N$ matrix $G$,
\begin{equation}
C_{1}\leq C_{2}\Longrightarrow GC_{1}G\leq GC_{2}G. \label{C1 G C2}%
\end{equation}
For any symmetric positive matrix $C$, we can rewrite $C=M^{T}\Delta M$ with
$M$ orthogonal and $\Delta=\operatorname*{diag}\left(  \lambda_{1}%
,...,\lambda_{n}\right)  $. Letting $C^{1/2}=M^{T}\Delta^{1/2}M$, one can
check that $C^{1/2}$ is still symmetric positive, and $C^{1/2}C^{1/2}=I$.
Moreover, writing $C^{-1/2}=\left(  C^{-1}\right)  ^{1/2}$ we have%
\begin{align*}
C^{-1/2}  &  =M^{T}\Delta^{-1/2}M\\
C^{-1/2}CC^{-1/2}  &  =I.
\end{align*}
Then, applying (\ref{C1 G C2}) with $G=C_{1}^{-1/2}$ we get%
\[
I=C_{1}^{-1/2}C_{1}C_{1}^{-1/2}\leq C_{1}^{-1/2}C_{2}C_{1}^{-1/2}.
\]
Next, applying (\ref{C1 G C2}) to the last inequality with $G=\left(
C_{1}^{-1/2}C_{2}C_{1}^{-1/2}\right)  ^{-1/2}$ we get%
\begin{align*}
&  C_{1}^{1/2}C_{2}^{-1}C_{1}^{1/2}\\
&  =\left(  C_{1}^{-1/2}C_{2}C_{1}^{-1/2}\right)  ^{-1}=\left(  C_{1}%
^{-1/2}C_{2}C_{1}^{-1/2}\right)  ^{-1/2}\left(  C_{1}^{-1/2}C_{2}C_{1}%
^{-1/2}\right)  ^{-1/2}\\
&  \leq\left(  C_{1}^{-1/2}C_{2}C_{1}^{-1/2}\right)  ^{-1/2}\left(
C_{1}^{-1/2}C_{2}C_{1}^{-1/2}\right)  \left(  C_{1}^{-1/2}C_{2}C_{1}%
^{-1/2}\right)  ^{-1/2}=I.
\end{align*}
Finally, applying (\ref{C1 G C2}) to the last inequality with $G=C_{1}^{-1/2}$
we get%
\[
C_{2}^{-1}=C_{1}^{-1/2}\left(  C_{1}^{1/2}C_{2}^{-1}C_{1}^{1/2}\right)
C_{1}^{-1/2}\leq C_{1}^{-1/2}C_{1}^{-1/2}=C_{1}^{-1}%
\]
so the first statement is proved. To show the inequality on determinants, we
can write, since $C_{1}\leq C_{2}$,%
\[
C_{2}^{-1/2}C_{1}C_{2}^{-1/2}\leq I.
\]
Letting $M$ be an orthogonal matrix that diagonalizes $C_{2}^{-1/2}C_{1}%
C_{2}^{-1/2}$ we get%
\[
\operatorname{diag}\left(  \lambda_{1},...,\lambda_{n}\right)  =M^{T}%
C_{2}^{-1/2}C_{1}C_{2}^{-1/2}M\leq I
\]
which implies $0<\lambda_{i}\leq1$ for $i=1,2,...,n$ hence also%
\[
1\geq%
{\textstyle\prod\limits_{i=1}^{n}}
\lambda_{i}=\det\left(  M^{T}C_{2}^{-1/2}C_{1}C_{2}^{-1/2}M\right)
=\frac{\det C_{1}}{\det C_{2}},
\]
so we are done.
\end{proof}

Applying Propositions \ref{Prop Pol LM} and \ref{Prop fq C C0} we immediately
get the following:

\begin{proposition}
\label{Prop inverse matrix}For every $\xi\in\mathbb{R}^{N}$ and every
$t>t_{0}$ we have%
\begin{equation}
\nu^{-1}\xi^{T}C_{0}\left(  t-t_{0}\right)  ^{-1}\xi\geq\xi^{T}C\left(
t,t_{0}\right)  ^{-1}\xi\geq\nu\xi^{T}C_{0}\left(  t-t_{0}\right)  ^{-1}%
\xi\label{fq inverse}%
\end{equation}%
\begin{equation}
\nu^{-N}\det C_{0}\left(  t-t_{0}\right)  \geq\det C\left(  t,t_{0}\right)
\geq\nu^{N}\det C_{0}\left(  t-t_{0}\right)  \label{det C det C0}%
\end{equation}
for every $t>t_{0}$.
\end{proposition}

We are now in position to give the

\bigskip

\begin{proof}
[Proof of Thm. \ref{Thm comparison Gammas}]Recall that $C_{0}\left(  t\right)
$ is defined in (\ref{C_0}). From the definition of the matrix $C\left(
t,t_{0}\right)  $ one immediately reads that, letting $C_{\nu}\left(
t,t_{0}\right)  $ be the matrix corresponding to the operator $\mathcal{L}%
_{\nu}$, one has%
\begin{equation}
C_{\nu}\left(  t,t_{0}\right)  =\nu C_{0}\left(  t-t_{0}\right)  \label{C nu}%
\end{equation}
hence also
\begin{equation}
\det\left(  C_{\nu}\left(  t,t_{0}\right)  \right)  =\nu^{N}\det C_{0}\left(
t-t_{0}\right)  . \label{det Cnu}%
\end{equation}
From the explicit form of $\Gamma$ given in (\ref{Gamma}) we read that
whenever the matrix $A\left(  t\right)  $ is constant one has%
\[
\Gamma\left(  x,t;x_{0},t_{0}\right)  =\Gamma\left(  x-E\left(  t-t_{0}%
\right)  x_{0},t-t_{0};0,0\right)  ,
\]
in particular this relation holds for $\Gamma_{\nu}$. Then (\ref{Gamma}),
(\ref{C nu}), (\ref{det Cnu}) imply (\ref{G_nu polo0}). Therefore
(\ref{fq inverse}) and (\ref{det C det C0}) give:%
\begin{align*}
\Gamma\left(  x,t;x_{0},t_{0}\right)   &  =\frac{e^{-\left(  \frac{1}%
{4}\left(  x-E\left(  t-t_{0}\right)  x_{0}\right)  ^{T}C\left(
t,t_{0}\right)  ^{-1}\left(  x-E\left(  t-t_{0}\right)  x_{0}\right)  +\left(
t-t_{0}\right)  \operatorname*{Tr}B\right)  }}{\left(  4\pi\right)
^{N/2}\sqrt{\det C\left(  t,t_{0}\right)  }}\\
&  \leq\frac{e^{-\left(  \frac{\nu}{4}\left(  x-E\left(  t-t_{0}\right)
x_{0}\right)  ^{T}C_{0}\left(  t-t_{0}\right)  ^{-1}\left(  x-E\left(
t-t_{0}\right)  x_{0}\right)  +\left(  t-t_{0}\right)  \operatorname*{Tr}%
B\right)  }}{\left(  4\pi\right)  ^{N/2}\sqrt{\nu^{N}\det C_{0}\left(
t-t_{0}\right)  }}\\
&  =\frac{1}{\nu^{N}}\Gamma_{\nu^{-1}}\left(  x,t;x_{0},t_{0}\right)  .
\end{align*}
Analogously,%
\begin{align*}
\Gamma\left(  x,t;x_{0},t_{0}\right)   &  \geq\frac{\nu^{N/2}e^{-\left(
\frac{1}{4\nu}\left(  x-E\left(  t-t_{0}\right)  x_{0}\right)  ^{T}%
C_{0}\left(  t-t_{0}\right)  ^{-1}\left(  x-E\left(  t-t_{0}\right)
x_{0}\right)  +\left(  t-t_{0}\right)  \operatorname*{Tr}B\right)  }}{\left(
4\pi\right)  ^{N/2}\sqrt{\det C_{0}\left(  t-t_{0}\right)  }}\\
&  =\nu^{N}\Gamma_{\nu}\left(  x,t;x_{0},t_{0}\right)
\end{align*}
so we have (\ref{G G_nu}).
\end{proof}

As anticipated in the introduction, the above comparison result has further
useful consequences when combined with some results of \cite{LP}, where
$\Gamma_{\alpha}$ is compared with the fundamental solution of the
\textquotedblleft principal part operator\textquotedblright%
\ $\widetilde{\mathcal{L}}_{\alpha}$ having the same matrix $A=\alpha I_{q,N}$
and a simpler matrix $B$, actually the matrix obtained from (\ref{B})
annihilating all the $\ast$ blocks. This operator $\widetilde{\mathcal{L}%
}_{\alpha}$ is also $2$-homogeneous with respect to dilations and its matrix
$C_{0}\left(  t\right)  $ (which in the next statement is called $C_{0}^{\ast
}\left(  t\right)  $) has a simpler form, which gives a useful asymptotic
estimate for the matrix of $\mathcal{L}_{\alpha}$. Namely, the following holds:

\begin{proposition}
[Short-time asymptotics of the matrix $C_{0}\left(  t\right)  $]%
\label{Prop LP}(See \cite[(3.14), (3.9), (2.17)]{LP}) There exist integers
$1=\sigma_{1}\leq\sigma_{2}\leq...\leq\sigma_{N}=2\kappa+1$ (with $\kappa$ as
in (\ref{B})), a constant invertible $N\times N$ matrix $C_{0}^{\ast}\left(
1\right)  $ and a $N\times N$ diagonal matrix%
\[
D_{0}\left(  \lambda\right)  =\operatorname*{diag}\left(  \lambda^{\sigma_{1}%
},\lambda^{\sigma_{2}},...,\lambda^{\sigma_{N}}\right)
\]
such that the following holds. If we let
\[
C_{0}^{\ast}\left(  t\right)  =D_{0}\left(  t^{1/2}\right)  C_{0}^{\ast
}\left(  1\right)  D_{0}\left(  t^{1/2}\right)  ,
\]
so that%
\[
\det C_{0}^{\ast}\left(  t\right)  =c_{N}t^{Q}%
\]
where $Q=\sum_{i=1}^{N}\sigma_{i},$ then:%
\begin{align*}
\det C_{0}\left(  t\right)   &  =\det C_{0}^{\ast}\left(  t\right)  \left(
1+tO\left(  1\right)  \right)  \text{ as }t\rightarrow0^{+}\\
x^{T}C_{0}\left(  t\right)  ^{-1}x  &  =x^{T}C_{0}^{\ast}\left(  t\right)
^{-1}x\left(  1+tO\left(  1\right)  \right)  \text{ as }t\rightarrow0^{+}%
\end{align*}
where in the second equality $O\left(  1\right)  $ stands for a bounded
function on $\mathbb{R}^{N}\times(0,1]$.
\end{proposition}

The above result allows to prove the following more explicit upper bound on
$\Gamma$ for short times:

\begin{proposition}
\label{Prop lim t0}There exist constants $c,\delta\in\left(  0,1\right)  $
such that for $0<t_{0}-t\leq\delta$ and every $x,x_{0}\in\mathbb{R}^{N}$ we have:
\end{proposition}

\begin{equation}
\Gamma\left(  x,t;x_{0},t_{0}\right)  \leq\frac{1}{c\left(  t-t_{0}\right)
^{Q/2}}e^{-c\frac{\left\vert x-E\left(  t-t_{0}\right)  x_{0}\right\vert ^{2}%
}{t-t_{0}}}. \label{Upper Gamma t piccolo}%
\end{equation}

\begin{proof}
By (\ref{G G_nu}) and the properties of the fundamental solution when the
matrix $A\left(  t\right)  $ is constant, we can write:%
\begin{equation}
\Gamma\left(  x,t;x_{0},t_{0}\right)  \leq\nu^{-N}\Gamma_{\nu^{-1}}\left(
x-E\left(  t-t_{0}\right)  x_{0},t-t_{0};0,0\right)  . \label{G Gnu explicit}%
\end{equation}
On the other hand,%
\[
\Gamma_{\alpha}\left(  y,t;0,0\right)  =\frac{1}{\left(  4\pi\alpha\right)
^{N/2}\sqrt{\det C_{0}\left(  t\right)  }}e^{-\left(  \frac{1}{4\alpha}%
y^{T}C_{0}\left(  t\right)  ^{-1}y+t\operatorname*{Tr}B\right)  }%
\]
and by Proposition \ref{Prop LP} there exist $c,\delta\in\left(  0,1\right)  $
such that for $0<t\leq\delta$ and every $y\in\mathbb{R}^{N}$%
\begin{align*}
\det C_{0}\left(  t\right)   &  =\det C_{0}^{\ast}\left(  t\right)  \left(
1+tO\left(  1\right)  \right)  \geq c\det C_{0}^{\ast}\left(  t\right)
=c_{1}t^{Q}\\
y^{T}C_{0}\left(  t\right)  ^{-1}y  &  =y^{T}C_{0}^{\ast}\left(  t\right)
^{-1}y\left(  1+tO\left(  1\right)  \right)  \geq c\,y^{T}C_{0}^{\ast}\left(
t\right)  ^{-1}y\\
&  \geq c\left\vert D_{0}\left(  t^{-1/2}\right)  y\right\vert ^{2}%
=c\sum_{i=1}^{N}\frac{y_{i}^{2}}{t^{\sigma_{i}}}\geq c\frac{\left\vert
y\right\vert ^{2}}{t}.
\end{align*}
Hence%
\begin{align*}
\Gamma\left(  x,t;x_{0},t_{0}\right)   &  \leq\frac{1}{\left(  4\pi\nu\right)
^{N/2}\left(  t-t_{0}\right)  ^{Q/2}}e^{\frac{\nu}{4}\left\vert
\operatorname*{Tr}B\right\vert }e^{-\nu c\frac{\left\vert x-E\left(
t-t_{0}\right)  x_{0}\right\vert ^{2}}{t-t_{0}}}\\
&  =\frac{1}{c_{2}\left(  t-t_{0}\right)  ^{Q/2}}e^{-c_{2}\frac{\left\vert
x-E\left(  t-t_{0}\right)  x_{0}\right\vert ^{2}}{t-t_{0}}}.
\end{align*}

\end{proof}

\section{Properties of the fundamental solution and Cauchy
problem\label{sec properties gamma}}

\subsection{Regularity properties of $\Gamma$ and
asymptotics\label{sec regularity}}

In this section we will prove point (i) of Theorem \ref{Thm main}.

With reference to the explicit form of $\Gamma$ in (\ref{Gamma}), we start
noting that the elements of the matrix%
\[
E\left(  t-\sigma\right)  A\left(  \sigma\right)  E\left(  t-\sigma\right)
^{T}%
\]
are measurable and uniformly essentially bounded for $\left(  t,\sigma
,t_{0}\right)  $ varying in any region $H\leq t_{0}\leq\sigma\leq t\leq K$ for
fixed $H,K\in\mathbb{R}$. This implies that the matrix%
\[
C\left(  t,t_{0}\right)  =\int_{t_{0}}^{t}E\left(  t-\sigma\right)  A\left(
\sigma\right)  E\left(  t-\sigma\right)  ^{T}d\sigma
\]
is Lipschitz continuous with respect to $t$ and with respect to $t_{0}$ in any
region $H\leq t_{0}\leq t\leq K$ for fixed $H,K\in\mathbb{R}$. Moreover,
$C\left(  t,t_{0}\right)  $ and $\det C\left(  t,t_{0}\right)  $ are jointly
continuous in $\left(  t,t_{0}\right)  $. Recalling that, by Proposition
\ref{Prop fq C C0}, the matrix $C\left(  t,t_{0}\right)  $ is positive
definite for any $t>t_{0}$, we also have that $C\left(  t,t_{0}\right)  ^{-1}$
is Lipschitz continuous with respect to $t$ and with respect to $t_{0}$ in any
region $H\leq t_{0}+\delta\leq t\leq K$ for fixed $H,K\in\mathbb{R}$ and
$\delta>0$, and is jointly continuous in $\left(  t,t_{0}\right)  $ for
$t>t_{0}$.

From the explicit form of $\Gamma$ and the previous remarks we conclude that
$\Gamma\left(  x,t;x_{0},t_{0}\right)  $ is jointly continuous in $\left(
x,t;x_{0},t_{0}\right)  $ for $t>t_{0}$, smooth w.r.t. $x$ and $x_{0}$ for
$t>t_{0}$ and Lipschitz continuous with respect to $t$ and with respect to
$t_{0}$ in any region $H\leq t_{0}+\delta\leq t\leq K$ for fixed
$H,K\in\mathbb{R}$ and $\delta>0$.

Moreover, every derivative $\frac{\partial^{\alpha+\beta}\Gamma}{\partial
x^{\alpha}\partial^{\beta}x_{0}}$ is given by $\Gamma$ times a polynomial in
$\left(  x,x_{0}\right)  $ with coefficients Lipschitz continuous with respect
to $t$ and with respect to $t_{0}$ in any region $H\leq t_{0}+\varepsilon\leq
t\leq K$ for fixed $H,K\in\mathbb{R}$ and $\varepsilon>0,$ and jointly
continuous in $\left(  t,t_{0}\right)  $ for $t>t_{0}$.

In order to show that $\Gamma$ and $\frac{\partial^{\alpha+\beta}\Gamma
}{\partial x^{\alpha}\partial x_{0}^{\beta}}$ are jointly continuous in the
region $\mathbb{R}_{\ast}^{2N+2}$ (see (\ref{R star})) we also need to show
that these functions tend to zero as $\left(  x,t\right)  \rightarrow\left(
y,t_{0}^{+}\right)  $ and $y\neq x_{0}$.\ For $\Gamma$, this assertion follows
by Proposition \ref{Prop lim t0}: for $y\neq x_{0}$ and $\left(  x,t\right)
\rightarrow\left(  y,t_{0}^{+}\right)  $ we have
\[
\left\vert x-E\left(  t-t_{0}\right)  x_{0}\right\vert ^{2}\rightarrow
\left\vert y-x_{0}\right\vert ^{2}\neq0,
\]
hence
\[
\frac{1}{\left(  t-t_{0}\right)  ^{Q/2}}e^{-c_{2}\frac{\left\vert x-E\left(
t-t_{0}\right)  x_{0}\right\vert ^{2}}{t-t_{0}}}\rightarrow0
\]
and the same is true for $\Gamma\left(  x,t;x_{0},t_{0}\right)  .$

To prove the analogous assertion for $\frac{\partial^{\alpha+\beta}\Gamma
}{\partial x^{\alpha}\partial x_{0}^{\beta}}$ we first need to establish some
upper bounds for these derivatives, which will be useful several times in the following.

\begin{proposition}
\label{Prop comput der}For $t>s$, let $C\left(  t,s\right)  ^{-1}=\left\{
\gamma_{ij}\left(  t,s\right)  \right\}  _{i,j=1}^{N}$ , let
\[
C^{\prime}\left(  t,s\right)  =E\left(  t-s\right)  ^{T}C\left(  t,s\right)
^{-1}E\left(  t-s\right)
\]
and let $C^{\prime}\left(  t,s\right)  =\left\{  \gamma_{ij}^{\prime}\left(
t,s\right)  \right\}  _{i,j=1}^{N}.$ Then:

(i) For every $x,y\in\mathbb{R}^{N}$, every $t>s$, $k,h=1,2,...,N,$%
\begin{equation}
\partial_{x_{k}}\Gamma\left(  x,t;y,s\right)  =-\frac{1}{2}\Gamma\left(
x,t;y,s\right)  \cdot\sum_{i=1}^{N}\gamma_{ik}\left(  t,s\right)  \left(
x-E\left(  t-s\right)  y\right)  _{i} \label{der 1 x}%
\end{equation}%
\begin{align}
&  \partial_{x_{h}x_{k}}^{2}\Gamma\left(  x,t;y,s\right)  =\Gamma\left(
x,t;y,s\right)  \left(  \frac{1}{4}\left(  \sum_{i}\gamma_{ik}\left(
t,s\right)  \left(  x-E\left(  t-s\right)  y\right)  _{i}\right)  \cdot\right.
\label{der 2 x}\\
&  \cdot\left.  \left(  \sum_{j}\gamma_{jh}\left(  t,s\right)  \left(
x-E\left(  t-s\right)  y\right)  _{j}\right)  -\frac{1}{2}\gamma_{hk}\left(
t,s\right)  \right) \nonumber
\end{align}%
\begin{equation}
\partial_{y_{k}}\Gamma\left(  x,t;y,s\right)  =-\frac{1}{2}\Gamma\left(
x,t;y,s\right)  \cdot\sum_{i=1}^{N}\gamma_{ik}^{\prime}\left(  t,s\right)
\left(  y-E\left(  s-t\right)  x\right)  _{i} \label{der 1 y}%
\end{equation}%
\begin{align}
&  \partial_{y_{h}y_{k}}^{2}\Gamma\left(  x,t;y,s\right)  =\Gamma\left(
x,t;y,s\right)  \cdot\left(  \frac{1}{4}\left(  \sum_{i}\gamma_{ik}^{\prime
}\left(  t,s\right)  \left(  y-E\left(  s-t\right)  x\right)  _{i}\right)
\cdot\right. \label{der 2 y}\\
&  \cdot\left.  \left(  \sum_{j}\gamma_{jh}^{\prime}\left(  t,s\right)
\left(  y-E\left(  s-t\right)  x\right)  _{j}\right)  -\frac{1}{2}\gamma
_{hk}^{\prime}\left(  t,s\right)  \right)  .\nonumber
\end{align}

(ii) For every $n,m=0,1,2,...$ there exists $c>0$ such that for every
$x,y\in\mathbb{R}^{N}$, every $t>s$
\begin{align}
&  \sum_{\left\vert \alpha\right\vert \leq n,\left\vert \beta\right\vert \leq
m}\left\vert \partial_{x}^{\alpha}\partial_{y}^{\beta}\Gamma\left(
x,t;y,s\right)  \right\vert \nonumber\\
&  \leq c\Gamma\left(  x,t;y,s\right)  \cdot\left\{  1+\left\Vert C\left(
t,s\right)  ^{-1}\right\Vert +\left\Vert C\left(  t,s\right)  ^{-1}\right\Vert
^{n}\left\vert x-E\left(  t-s\right)  y\right\vert ^{n}\right\} \nonumber\\
&  \cdot\left\{  1+\left\Vert C^{\prime}\left(  t,s\right)  \right\Vert
+\left\Vert C^{\prime}\left(  t,s\right)  \right\Vert ^{m}\left\vert
y-E\left(  s-t\right)  x\right\vert ^{m}\right\}  \label{higher order der G}%
\end{align}
where $\left\Vert \cdot\right\Vert $ stands for a matrix norm.
\end{proposition}

\begin{proof}
A straightforward computation gives (\ref{der 1 x}) and (\ref{der 2 x}).
Iterating this computation we can also bound%
\begin{align*}
&  \sum_{\left\vert \alpha\right\vert \leq n}\left\vert \partial_{x}^{\alpha
}\Gamma\left(  x,t;y,s\right)  \right\vert \leq c\Gamma\left(  x,t;y,s\right)
\cdot\\
&  \cdot\left\{  1+\left\Vert C\left(  t,s\right)  ^{-1}\right\Vert
+\left\Vert C\left(  t,s\right)  ^{-1}\right\Vert ^{n}\left\vert x-E\left(
t-s\right)  y\right\vert ^{n}\right\}  .
\end{align*}

To compute $y$-derivatives of $\Gamma$, it is convenient to write%
\begin{align*}
&  \left(  x-E\left(  t-s\right)  y\right)  ^{T}C\left(  t,s\right)
^{-1}\left(  x-E\left(  t-s\right)  y\right) \\
&  =\left(  y-E\left(  s-t\right)  x\right)  ^{T}C^{\prime}\left(  t,s\right)
\left(  y-E\left(  s-t\right)  x\right)
\end{align*}
with%
\[
C^{\prime}\left(  t,s\right)  =E\left(  t-s\right)  ^{T}C\left(  t,s\right)
^{-1}E\left(  t-s\right)  .
\]
With this notation, we have%
\[
\Gamma\left(  x,t;y,s\right)  =\frac{1}{\left(  4\pi\right)  ^{N/2}\sqrt{\det
C\left(  t,s\right)  }}e^{-\left(  \frac{1}{4}\left(  y-E\left(  s-t\right)
x\right)  ^{T}C^{\prime}\left(  t,s\right)  \left(  y-E\left(  s-t\right)
x\right)  +\left(  t-s\right)  \operatorname*{Tr}B\right)  }%
\]
and an analogous computation gives (\ref{der 1 y}), (\ref{der 2 y}) and, by
iteration%
\begin{align*}
&  \sum_{\left\vert \alpha\right\vert \leq m}\left\vert \partial_{y}^{\alpha
}\Gamma\left(  x,t;y,s\right)  \right\vert \leq c\Gamma\left(  x,t;y,s\right)
\cdot\\
&  \cdot\left\{  1+\left\Vert C^{\prime}\left(  t,s\right)  \right\Vert
+\left\Vert C^{\prime}\left(  t,s\right)  \right\Vert ^{m}\left\vert
y-E\left(  s-t\right)  x\right\vert ^{m}\right\}
\end{align*}
and finally also (\ref{higher order der G}).
\end{proof}

With the previous bounds in hands we can now prove the following:

\begin{theorem}
[Upper bounds on the derivatives of $\Gamma$]\label{Thm bounds derivatives}(i)
For every $n,m=0,1,2...$ and $t,s$ ranging in a compact subset of $\left\{
\left(  t,s\right)  :t\geq s+\varepsilon\right\}  $ for some $\varepsilon>0$
we have%
\begin{align}
&  \sum_{\left\vert \alpha\right\vert \leq n,\left\vert \beta\right\vert \leq
m}\left\vert \partial_{x}^{\alpha}\partial_{y}^{\beta}\Gamma\left(
x,t;y,s\right)  \right\vert \label{global bounds der}\\
&  \leq Ce^{-C^{\prime}\left\vert x-E\left(  t-s\right)  y\right\vert ^{2}%
}\cdot\left\{  1+\left\vert x-E\left(  t-s\right)  y\right\vert ^{n}%
+\left\vert y-E\left(  s-t\right)  x\right\vert ^{m}\right\} \nonumber
\end{align}
for every $x,y\in\mathbb{R}^{N}$, for constants $C,C^{\prime}$ depending on
$n,m$ and the compact set.

In particular, for fixed $t>s$ we have%
\begin{align*}
\lim_{\left\vert x\right\vert \rightarrow+\infty}\partial_{x}^{\alpha}%
\partial_{y}^{\beta}\Gamma\left(  x,t;y,s\right)   &  =0\text{ for every }%
y\in\mathbb{R}^{N}\\
\lim_{\left\vert y\right\vert \rightarrow+\infty}\partial_{x}^{\alpha}%
\partial_{y}^{\beta}\Gamma\left(  x,t;y,s\right)   &  =0\text{ for every }%
x\in\mathbb{R}^{N}%
\end{align*}
for every multiindices $\alpha,\beta$.

(ii) For every $n,m=0,1,2...$ there exists $\delta\in\left(  0,1\right)
,C,c>0$ such that for $0<t-s<\delta$ and every $x,y\in\mathbb{R}^{N}$ we have%
\begin{align}
&  \sum_{\left\vert \alpha\right\vert \leq n,\left\vert \beta\right\vert \leq
m}\left\vert \partial_{x}^{\alpha}\partial_{y}^{\beta}\Gamma\left(
x,t;y,s\right)  \right\vert \nonumber\\
&  \leq\frac{C}{\left(  t-s\right)  ^{Q/2}}e^{-c\frac{\left\vert x-E\left(
t-s\right)  y\right\vert ^{2}}{t-s}}\cdot\left\{  \left(  t-s\right)
^{-\sigma_{N}}+\left(  t-s\right)  ^{-n\sigma_{N}}\left\vert x-E\left(
t-s\right)  y\right\vert ^{n}\right\} \nonumber\\
&  \cdot\left\{  \left(  t-s\right)  ^{-\sigma_{N}}+\left(  t-s\right)
^{-m\sigma_{N}}\left\vert y-E\left(  s-t\right)  x\right\vert ^{m}\right\}  .
\label{bound der small t}%
\end{align}

In particular, for every fixed $x_{0},y\in\mathbb{R}^{N},$ $x_{0}\neq
y,s\in\mathbb{R},$%
\[
\lim_{\left(  x,t\right)  \rightarrow\left(  x_{0},s^{+}\right)  }%
\sum_{\left\vert \alpha\right\vert +\left\vert \beta\right\vert \leq
k}\left\vert \partial_{x}^{\alpha}\partial_{y}^{\beta}\Gamma\left(
x,t;y,s\right)  \right\vert =0
\]
so that $\Gamma$ and $\partial_{x}^{\alpha}\partial_{y}^{\beta}\Gamma\left(
x,t;y,s\right)  $ are jointly continuous in the region $\mathbb{R}_{\ast
}^{2N+2}.$
\end{theorem}

\begin{proof}
(i) The matrix $C\left(  t,s\right)  $ is jointly continuous in $\left(
t,s\right)  $ and, by Proposition \ref{Prop fq C C0} is positive definite for
any $t>s$. Hence for $t,s$ ranging in a compact subset of $\left\{  \left(
t,s\right)  :t\geq s+\varepsilon\right\}  $ we have
\begin{align*}
\left\Vert C\left(  t,s\right)  ^{-1}\right\Vert ^{n}+\left\Vert C^{\prime
}\left(  t,s\right)  \right\Vert ^{m}  &  \leq c\\
e^{-\left(  \frac{1}{4}\left(  x-E\left(  t-s\right)  y\right)  ^{T}C\left(
t,s\right)  ^{-1}\left(  x-E\left(  t-s\right)  y\right)  +\left(  t-s\right)
\operatorname*{Tr}B\right)  }  &  \leq c_{1}e^{-c\left\vert x-E\left(
t-s\right)  y\right\vert ^{2}}%
\end{align*}
for some $c,c_{1}>0$ only depending on $n,m$ and the compact set. Hence by
(\ref{higher order der G}) and (\ref{Gamma}) we get (\ref{global bounds der}).

Let now $t,s$ be fixed. If $y$ is fixed and $\left\vert x\right\vert
\rightarrow\infty$ then (\ref{global bounds der}) gives%
\[
\sum_{\left\vert \alpha\right\vert \leq n,\left\vert \beta\right\vert \leq
m}\left\vert \partial_{x}^{\alpha}\partial_{y}^{\beta}\Gamma\left(
x,t;y,s\right)  \right\vert \leq Ce^{-C^{\prime}\left\vert x\right\vert ^{2}%
}\left\{  1+\left\vert x\right\vert ^{n}+\left\vert x\right\vert ^{m}\right\}
\rightarrow0.
\]
If $x$ is fixed and $\left\vert y\right\vert \rightarrow\infty,$%
\begin{align*}
&  \sum_{\left\vert \alpha\right\vert \leq n,\left\vert \beta\right\vert \leq
m}\left\vert \partial_{x}^{\alpha}\partial_{y}^{\beta}\Gamma\left(
x,t;y,s\right)  \right\vert \\
&  \leq Ce^{-C^{\prime}\left\vert E\left(  t-s\right)  y\right\vert ^{2}%
}\left\{  1+\left\vert E\left(  t-s\right)  y\right\vert ^{n}+\left\vert
E\left(  s-t\right)  x\right\vert ^{m}\right\}  \rightarrow0,
\end{align*}
because when $\left\vert y\right\vert \rightarrow\infty$ also $\left\vert
E\left(  t-s\right)  y\right\vert \rightarrow\infty$, since $E\left(
t-s\right)  $ is invertible.

(ii) Applying (\ref{higher order der G}) together with Proposition
\ref{Prop lim t0} we get that for some $\delta\in\left(  0,1\right)  $,
whenever $0<t-s<\delta$ we have%
\begin{align*}
&  \sum_{\left\vert \alpha\right\vert \leq n,\left\vert \beta\right\vert \leq
m}\left\vert \partial_{x}^{\alpha}\partial_{y}^{\beta}\Gamma\left(
x,t;y,s\right)  \right\vert \\
&  \leq\frac{1}{c\left(  t-s\right)  ^{Q/2}}e^{-c\frac{\left\vert x-E\left(
t-s\right)  y\right\vert ^{2}}{t-s}}\cdot\left\{  1+\left\Vert C\left(
t,s\right)  ^{-1}\right\Vert +\left\Vert C\left(  t,s\right)  ^{-1}\right\Vert
^{n}\left\vert x-E\left(  t-s\right)  y\right\vert ^{n}\right\} \\
&  \cdot\left\{  1+\left\Vert C^{\prime}\left(  t,s\right)  \right\Vert
+\left\Vert C^{\prime}\left(  t,s\right)  \right\Vert ^{m}\left\vert
y-E\left(  s-t\right)  x\right\vert ^{m}\right\}  .
\end{align*}
Next, we recall that by Proposition \ref{Prop inverse matrix} we have%
\[
\left\Vert C\left(  t,s\right)  ^{-1}\right\Vert \leq c\left\Vert C_{0}\left(
t-s\right)  ^{-1}\right\Vert
\]
by Proposition \ref{Prop LP}, for $0<t-s\leq\delta$
\[
\leq c^{\prime}\left\Vert C_{0}^{\ast}\left(  t-s\right)  ^{-1}\right\Vert
\leq c^{\prime\prime}\left(  t-s\right)  ^{-\sigma_{N}}%
\]
and an analogous bound holds for $C^{\prime}\left(  t,s\right)  $, for small
$\left(  t-s\right)  $. Hence we get (\ref{bound der small t}).

If now $x_{0}\neq y$ are fixed, from (\ref{bound der small t}) we deduce%
\[
\sum_{\left\vert \alpha\right\vert \leq n,\left\vert \beta\right\vert \leq
m}\left\vert \partial_{x}^{\alpha}\partial_{y}^{\beta}\Gamma\left(
x,t;y,s\right)  \right\vert \leq\frac{C}{\left(  t-s\right)  ^{\frac{Q}%
{2}+\left(  n+m\right)  \sigma_{N}}}\exp\left(  -\frac{c}{t-s}\right)
\rightarrow0
\]
as $\left(  x,t\right)  \rightarrow\left(  x_{0},s^{+}\right)  .$
\end{proof}

With the above theorem, the proof of point (i) in Theorem \ref{Thm main} is complete.

\begin{remark}
[Long time behavior of $\Gamma$]We have shown that the fundamental solution
$\Gamma\left(  x,t;y,s\right)  $ and its spacial derivatives of every order
tend to zero for $x$ or $y$ going to infinity, and tend to zero for
$t\rightarrow s^{+}$ and $x\neq y$. It is natural to ask what happens for
$t\rightarrow+\infty$. However, nothing can be said in general about this
limit, even when the coefficients $a_{ij}$ are constant, and even in
nondegenerate cases.\ Compare, for $N=1$, the heat operator%
\[
Hu=u_{xx}-u_{t},
\]
for which%
\[
\Gamma\left(  x,y;0,0\right)  =\frac{1}{\sqrt{4\pi t}}e^{-\frac{x^{2}}{4t}%
}\rightarrow0\text{ for }t\rightarrow+\infty\text{, every }x\in\mathbb{R}%
\]
and the operator%
\[
Lu=u_{xx}+xu_{x}-u_{t}%
\]
for which (\ref{Gamma}) gives%
\[
\Gamma\left(  x,t;0,0\right)  =\frac{1}{\sqrt{2\pi\left(  1-e^{-2t}\right)  }%
}e^{-\frac{x^{2}}{2\left(  1-e^{-2t}\right)  }}\rightarrow\frac{1}{\sqrt{2\pi
}}e^{-\frac{x^{2}}{2}}\text{ as }t\rightarrow+\infty.
\]

\end{remark}

\subsection{$\Gamma$ is a solution\label{sec solves equation}}

In this section we will prove points (ii), (iii), (vi) of Theorem
\ref{Thm main}.

We want to check that our \textquotedblleft candidate fundamental
solution\textquotedblright\ with pole at $\left(  x_{0},t_{0}\right)  $, given
by (\ref{Gamma}), actually solves the equation outside the pole, with respect
to $\left(  x,t\right)  $. Note that, by the results in
\S \ \ref{sec regularity} we already know that $\Gamma$ is infinitely
differentiable w.r.t. $x,x_{0}$, and a.e. differentiable w.r.t. $t,t_{0}$.

\begin{theorem}
\label{Thm LG=0}For every fixed $\left(  x_{0},t_{0}\right)  \in
\mathbb{R}^{N+1}$,
\[
\mathcal{L}\left(  \Gamma\left(  \cdot,\cdot;x_{0},t_{0}\right)  \right)
\left(  x,t\right)  =0\text{ for a.e. }t>t_{0}\text{ and every }x\in
\mathbb{R}^{N}\text{.}%
\]

\end{theorem}

Before proving the theorem, let us establish the following easy fact, which
will be useful in the subsequent computation and is also interesting in its own:

\begin{proposition}
\label{Prop integral Gamma dx}For every $t>t_{0}$ and $x_{0}\in\mathbb{R}^{N}$
we have
\begin{align}
\int_{\mathbb{R}^{N}}\Gamma\left(  x,t;x_{0},t_{0}\right)  dx  &  =e^{-\left(
t-t_{0}\right)  \operatorname*{Tr}B}\label{int Gamma tr}\\
\int_{\mathbb{R}^{N}}\Gamma\left(  x_{0},t;y,t_{0}\right)  dy  &  =1.\nonumber
\end{align}

\end{proposition}

\begin{proof}
Let us compute, for $t>t_{0}$:%
\begin{align*}
&  \int_{\mathbb{R}^{N}}\Gamma\left(  x,t;x_{0},t_{0}\right)  dx\\
&  =\frac{e^{-\left(  t-t_{0}\right)  \operatorname*{Tr}B}}{\left(
4\pi\right)  ^{N/2}\sqrt{\det C\left(  t,t_{0}\right)  }}\int_{\mathbb{R}^{N}%
}e^{-\frac{1}{4}\left(  x-E\left(  t-t_{0}\right)  x_{0}\right)  ^{T}%
C^{-1}\left(  t,t_{0}\right)  \left(  x-E\left(  t-t_{0}\right)  x_{0}\right)
}dx\\
\text{letting }x  &  =E\left(  t-t_{0}\right)  x_{0}+2C\left(  t,t_{0}\right)
^{1/2}y;dx=2^{N}\det C\left(  t,t_{0}\right)  ^{1/2}dy\\
&  =\frac{e^{-\left(  t-t_{0}\right)  \operatorname*{Tr}B}}{\left(
4\pi\right)  ^{N/2}\sqrt{\det C\left(  t,t_{0}\right)  }}2^{N}\sqrt{\det
C\left(  t,t_{0}\right)  }\int_{\mathbb{R}^{N}}e^{-\left\vert y\right\vert
^{2}}dy=e^{-\left(  t-t_{0}\right)  \operatorname*{Tr}B}.
\end{align*}
Next,%
\begin{align*}
&  \int_{\mathbb{R}^{N}}\Gamma\left(  x_{0},t;y,t_{0}\right)  dy\\
&  =\frac{e^{-\left(  t-t_{0}\right)  \operatorname*{Tr}B}}{\left(
4\pi\right)  ^{N/2}\sqrt{\det C\left(  t,t_{0}\right)  }}\int_{\mathbb{R}^{N}%
}e^{-\frac{1}{4}\left(  x_{0}-E\left(  t-t_{0}\right)  y\right)  ^{T}%
C^{-1}\left(  t,t_{0}\right)  \left(  x_{0}-E\left(  t-t_{0}\right)  y\right)
}dy\\
\text{letting }y  &  =E\left(  t_{0}-t\right)  \left(  x_{0}-2C\left(
t,t_{0}\right)  ^{1/2}z\right)  ;\\
dy  &  =2^{N}\det C\left(  t,t_{0}\right)  ^{1/2}\det E\left(  t_{0}-t\right)
dz=2^{N}\det C\left(  t,t_{0}\right)  ^{1/2}e^{\left(  t-t_{0}\right)
\operatorname{Tr}B}dz
\end{align*}%
\[
=\frac{e^{-\left(  t-t_{0}\right)  \operatorname*{Tr}B}}{\left(  4\pi\right)
^{N/2}\sqrt{\det C\left(  t,t_{0}\right)  }}2^{N}\det C\left(  t,t_{0}\right)
^{1/2}e^{\left(  t-t_{0}\right)  \operatorname{Tr}B}\int_{\mathbb{R}^{N}%
}e^{-\left\vert y\right\vert ^{2}}dy=1.
\]
Here in the change of variables we used the relation $\det\left(  \exp
B\right)  =e^{\operatorname{Tr}B}$, holding for every square matrix $B$.
\end{proof}

\bigskip

\begin{proof}
[Proof of Theorem \ref{Thm LG=0}]Keeping the notation of Proposition
\ref{Prop comput der}, and exploiting (\ref{der 1 x})-(\ref{der 2 x}) we have%
\begin{align}
&  \sum_{k,j=1}^{N}b_{jk}x_{k}\partial_{x_{j}}\Gamma\left(  x,t;x_{0}%
,t_{0}\right)  =\left(  \nabla_{x}\Gamma\left(  x,t;x_{0},t_{0}\right)
\right)  ^{T}Bx\nonumber\\
&  =-\frac{1}{2}\Gamma\left(  x,t;x_{0},t_{0}\right)  \left(  x-E\left(
t-t_{0}\right)  x_{0}\right)  ^{T}C\left(  t,t_{0}\right)  ^{-1}Bx.
\label{xBDG}%
\end{align}%
\begin{align}
&  \sum_{h,k=1}^{q}a_{hk}\left(  t\right)  \partial_{x_{h}x_{k}}^{2}%
\Gamma\left(  x,t;x_{0},t_{0}\right) \nonumber\\
&  =\Gamma\left\{  \frac{1}{4}\sum_{i,j=1}^{N}\left(  \sum_{h,k=1}^{q}%
a_{hk}\left(  t\right)  \gamma_{ik}\left(  t,t_{0}\right)  \gamma_{jh}\left(
t\right)  \right)  \cdot\right. \nonumber\\
&  \left.  \cdot\left(  x-E\left(  t-t_{0}\right)  x_{0}\right)  _{i}\left(
x-E\left(  t-t_{0}\right)  x_{0}\right)  _{j}-\frac{1}{2}\sum_{h,k=1}%
^{q}a_{hk}\left(  t\right)  \gamma_{hk}\left(  t,t_{0}\right)  \right\}
\nonumber\\
&  =\Gamma\left(  x,t;x_{0},t_{0}\right)  \cdot\left\{  \frac{1}{4}\left(
x-E\left(  t-t_{0}\right)  x_{0}\right)  ^{T}C\left(  t,t_{0}\right)
^{-1}A\left(  t\right)  C^{-1}\left(  x-E\left(  t-t_{0}\right)  x_{0}\right)
\right. \label{AGxx}\\
&  \left.  -\frac{1}{2}\operatorname{Tr}A\left(  t\right)  C\left(
t,t_{0}\right)  ^{-1}\right\}  .\nonumber
\end{align}%
\begin{align}
&  \partial_{t}\Gamma\left(  x,t;x_{0},t_{0}\right) \nonumber\\
&  =-\frac{\partial_{t}\left(  \det C\left(  t,t_{0}\right)  \right)
}{\left(  4\pi\right)  ^{N/2}2\det^{3/2}C\left(  t,t_{0}\right)  }e^{-\left(
\frac{1}{4}\left(  x-E\left(  t-t_{0}\right)  x_{0}\right)  ^{T}C\left(
t,t_{0}\right)  ^{-1}\left(  x-E\left(  t-t_{0}\right)  x_{0}\right)  +\left(
t-t_{0}\right)  \operatorname*{Tr}B\right)  }\nonumber\\
&  -\Gamma\left(  x,t;x_{0},t_{0}\right)  \cdot\nonumber\\
&  \cdot\partial_{t}\left(  \frac{1}{4}\left(  x-E\left(  t-t_{0}\right)
x_{0}\right)  ^{T}C\left(  t,t_{0}\right)  ^{-1}\left(  x-E\left(
t-t_{0}\right)  x_{0}\right)  +\left(  t-t_{0}\right)  \operatorname*{Tr}%
B\right) \nonumber\\
=  &  -\Gamma\left(  x,t;x_{0},t_{0}\right)  \left\{  \frac{\partial
_{t}\left(  \det C\left(  t,t_{0}\right)  \right)  }{2\det C\left(
t,t_{0}\right)  }\right. \label{D_tGamma}\\
&  +\left.  \frac{1}{4}\partial_{t}\left(  \left(  x-E\left(  t-t_{0}\right)
x_{0}\right)  ^{T}C\left(  t,t_{0}\right)  ^{-1}\left(  x-E\left(
t-t_{0}\right)  x_{0}\right)  \right)  +\operatorname*{Tr}B\right\}
.\nonumber
\end{align}
To shorten notation, from now on, throughout this proof, we will write%
\begin{align*}
&  C\text{ for }C\left(  t,t_{0}\right)  \text{, and}\\
&  E\text{ for }E\left(  t-t_{0}\right)  .
\end{align*}
To compute the $t$-derivative appearing in (\ref{D_tGamma}) we start writing%
\begin{align}
&  \partial_{t}\left(  \left(  x-Ex_{0}\right)  ^{T}C^{-1}\left(
x-Ex_{0}\right)  \right) \nonumber\\
&  =2\left(  -\partial_{t}Ex_{0}\right)  ^{T}C^{-1}\left(  x-Ex_{0}\right)
\nonumber\\
&  +\left(  x-Ex_{0}\right)  ^{T}\partial_{t}\left(  C^{-1}\right)  \left(
x-Ex_{0}\right)  . \label{D_t()}%
\end{align}
First, we note that%
\begin{equation}
\partial_{t}E=-B\exp\left(  -\left(  t-t_{0}\right)  B\right)  =-BE.
\label{D_tE}%
\end{equation}
Also, note that $B$ commutes with $E\left(  t\right)  $ and $B^{T}$ commutes
with $E\left(  t\right)  ^{T}$. Second, differentiating the identity
$C^{.-1}C=I$ we get%
\begin{equation}
\partial_{t}\left(  C^{-1}\right)  =-C^{-1}\partial_{t}\left(  C\right)
C^{-1}. \label{D_t(C^-1)}%
\end{equation}
In turn, at least for a.e. $t$, we have
\begin{align*}
\partial_{t}\left(  C\left(  t,t_{0}\right)  \right)   &  =E\left(  0\right)
A\left(  t\right)  E\left(  0\right)  ^{T}+\int_{t_{0}}^{t}\partial
_{t}E\left(  t-\sigma\right)  A\left(  \sigma\right)  E\left(  t-\sigma
\right)  ^{T}d\sigma\\
&  +\int_{t_{0}}^{t}E\left(  t-\sigma\right)  A\left(  \sigma\right)
\partial_{t}E\left(  t-\sigma\right)  ^{T}d\sigma\\
&  =A\left(  t\right)  -BC-CB^{T}.
\end{align*}
By (\ref{D_t(C^-1)}) this gives%
\begin{equation}
\partial_{t}\left(  C^{-1}\right)  =-C^{-1}A\left(  t\right)  C^{-1}%
+C^{-1}B+B^{T}C^{-1}. \label{D_t(C^-1)2}%
\end{equation}
Inserting (\ref{D_tE}) and (\ref{D_t(C^-1)2}) in (\ref{D_t()}) and then in
(\ref{D_tGamma}) we have%
\begin{align*}
&  \partial_{t}\left(  \left(  x-Ex_{0}\right)  ^{T}C^{-1}\left(
x-Ex_{0}\right)  \right) \\
&  =2\left(  BEx_{0}\right)  ^{T}C^{-1}\left(  x-Ex_{0}\right) \\
&  +\left(  x-Ex_{0}\right)  ^{T}\left[  -C^{-1}A\left(  t\right)
C^{-1}+2B^{T}C^{-1}\right]  \left(  x-Ex_{0}\right)  .
\end{align*}%
\begin{align}
\partial_{t}\Gamma &  =-\Gamma\left\{  \frac{\partial_{t}\left(  \det
C\right)  }{2\det C}+\operatorname*{Tr}B\right.  +\frac{1}{4}\left[  2\left(
BEx_{0}\right)  ^{T}C^{-1}\left(  x-Ex_{0}\right)  \right. \nonumber\\
&  \left.  \frac{{}}{{}}+\left.  \left(  x-Ex_{0}\right)  ^{T}\left[
-C^{-1}A\left(  t\right)  C^{-1}+2B^{T}C^{-1}\right]  \left(  x-Ex_{0}\right)
\right]  \right\} \nonumber
\end{align}%
\begin{align}
&  =-\Gamma\left\{  \frac{\partial_{t}\left(  \det C\right)  }{2\det
C}+\operatorname*{Tr}B\right.  -\frac{1}{4}\left(  x-Ex_{0}\right)  ^{T}%
C^{-1}A\left(  t\right)  C^{-1}\left(  x-Ex_{0}\right) \nonumber\\
&  +\left.  \frac{1}{2}x^{T}B^{T}C^{-1}\left(  x-Ex_{0}\right)  \right\}  .
\label{D_tGlast}%
\end{align}

Exploiting (\ref{AGxx}), (\ref{xBDG}) and (\ref{D_tGlast}) we can now compute
$\mathcal{L}\Gamma$:
\begin{align*}
&  \sum_{h,k=1}^{q}a_{hk}\left(  t\right)  \partial_{x_{h}x_{k}}^{2}%
\Gamma+\left(  \nabla\Gamma\right)  ^{T}Bx-\partial_{t}\Gamma\\
&  =\Gamma\left\{  \frac{1}{4}\left(  x-Ex_{0}\right)  ^{T}C^{-1}A\left(
t\right)  C^{-1}\left(  x-Ex_{0}\right)  -\frac{1}{2}\operatorname{Tr}A\left(
t\right)  C^{-1}\right. \\
&  -\frac{1}{2}\Gamma\left(  x-Ex_{0}\right)  ^{T}C^{-1}Bx+\frac{\partial
_{t}\left(  \det C\right)  }{2\det C}+\operatorname*{Tr}B\\
&  -\frac{1}{4}\left(  x-Ex_{0}\right)  ^{T}C^{-1}A\left(  t\right)
C^{-1}\left(  x-Ex_{0}\right)  +\left.  \frac{1}{2}x^{T}B^{T}C^{-1}\left(
x-Ex_{0}\right)  \right\} \\
&  =\Gamma\left\{  -\frac{1}{2}\operatorname{Tr}A\left(  t\right)
C^{-1}+\frac{\partial_{t}\left(  \det C\right)  }{2\det C}+\operatorname*{Tr}%
B\right\}  .
\end{align*}
To conclude our proof we are left to check that, in the last expression, the
quantity in braces identically vanishes for $t>t_{0}$. This, however, is not a
straightforward computation, since the term $\partial_{t}\left(  \det
C\right)  $ is not easily explicitly computed. Let us state this fact as a
separate ancillary result.
\end{proof}

\begin{proposition}
\label{Prop Ddet}For a.e. $t>t_{0}$ we have%
\[
\frac{\partial_{t}\left(  \det C\left(  t,t_{0}\right)  \right)  }{2\det
C\left(  t,t_{0}\right)  }=\frac{1}{2}\operatorname*{Tr}A\left(  t\right)
C\left(  t,t_{0}\right)  ^{-1}-\operatorname*{Tr}B.
\]

\end{proposition}

To prove this proposition we also need the following

\begin{lemma}
\label{Lemma Gauss xAx}For every $N\times N$ matrix $A$, and every $x_{0}%
\in\mathbb{R}^{N}$ we have:%
\begin{align}
\int_{\mathbb{R}^{N}}e^{-\left\vert x\right\vert ^{2}}\left(  x^{T}Ax\right)
dx  &  =\frac{\pi^{N/2}}{2}\operatorname*{Tr}A\nonumber\\
\int_{\mathbb{R}^{N}}e^{-\left\vert x\right\vert ^{2}}\left(  x_{0}%
^{T}Ax\right)  dx  &  =0. \label{null Gaussian}%
\end{align}

\end{lemma}

\begin{proof}
[Proof of Lemma \ref{Lemma Gauss xAx}]The second identity is obvious for
symmetry reasons. As to the first one, letting $A=\left(  a_{ij}\right)
_{i,j=1}^{N},$%
\begin{align*}
&  \int_{\mathbb{R}^{N}}e^{-\left\vert x\right\vert ^{2}}\left(
x^{T}Ax\right)  dx\\
&  =\sum_{i=1}^{N}\left\{  \sum_{j=1,...,N,j\neq i}a_{ij}\int_{\mathbb{R}^{N}%
}e^{-\left\vert x\right\vert ^{2}}x_{i}x_{j}dx+a_{ii}\int_{\mathbb{R}^{N}%
}e^{-\left\vert x\right\vert ^{2}}x_{i}^{2}dx\right\} \\
&  =\sum_{i=1}^{N}\left\{  0+a_{ii}\left(  \int_{\mathbb{R}^{N-1}%
}e^{-\left\vert w\right\vert ^{2}}dw\right)  \left(  \int_{\mathbb{R}%
}e^{-x_{i}^{2}}x_{i}^{2}dx_{i}\right)  \right\} \\
&  =\sum_{i=1}^{N}a_{ii}\pi^{\frac{N-1}{2}}\left(  \int_{\mathbb{R}}e^{-t^{2}%
}t^{2}dt\right)  =\pi^{\frac{N-1}{2}}\cdot\frac{\sqrt{\pi}}{2}\sum_{i=1}%
^{N}a_{ii}=\frac{\pi^{N/2}}{2}\operatorname*{Tr}A
\end{align*}
where the integrals corresponding to the terms with $i\neq j$ vanish for
symmetry reasons.
\end{proof}

\bigskip

\begin{proof}
[Proof of Proposition \ref{Prop Ddet}]Taking $\frac{\partial}{\partial t}$ in
the identity (\ref{int Gamma tr}) we have, by (\ref{D_tGlast}), for almost
every $t>t_{0},$%
\begin{align*}
&  -e^{-\left(  t-t_{0}\right)  \operatorname*{Tr}B}\operatorname*{Tr}%
B=\int_{\mathbb{R}^{N}}\frac{\partial\Gamma}{\partial t}\left(  x,t;x_{0}%
,t_{0}\right)  dx\\
&  =-\int_{\mathbb{R}^{N}}\Gamma\left(  x,t;x_{0},t_{0}\right)  \left\{
\frac{\partial_{t}\left(  \det C\right)  }{2\det C}+\operatorname*{Tr}%
B+\frac{1}{2}x^{T}B^{T}C^{-1}\left(  x-Ex_{0}\right)  \right. \\
&  \left.  -\frac{1}{4}\left(  x-Ex_{0}\right)  ^{T}C^{-1}A\left(  t\right)
C^{-1}\left(  x-Ex_{0}\right)  \right\}  dx
\end{align*}%
\begin{align*}
&  =-\left\{  \frac{\partial_{t}\left(  \det C\right)  }{2\det C}%
+\operatorname*{Tr}B\right\}  e^{-\left(  t-t_{0}\right)  \operatorname*{Tr}%
B}-\int_{\mathbb{R}^{N}}\Gamma\left(  x,t;x_{0},t_{0}\right)  \left\{
\frac{1}{2}x^{T}B^{T}C^{-1}\left(  x-Ex_{0}\right)  \right. \\
&  \left.  -\frac{1}{4}\left(  x-Ex_{0}\right)  ^{T}C^{-1}A\left(  t\right)
C^{-1}\left(  x-Ex_{0}\right)  \right\}  dx
\end{align*}
hence%
\begin{align*}
&  \frac{\partial_{t}\left(  \det C\right)  }{2\det C}\cdot e^{-\left(
t-t_{0}\right)  \operatorname*{Tr}B}=-\frac{e^{-\left(  t-t_{0}\right)
\operatorname*{Tr}B}}{\left(  4\pi\right)  ^{N/2}\sqrt{\det C}}\cdot\\
&  \cdot\int_{\mathbb{R}^{N}}e^{-\frac{1}{4}\left(  x-Ex_{0}\right)
^{T}C^{-1}\left(  x-Ex_{0}\right)  }\left\{  \frac{1}{2}x^{T}B^{T}%
C^{-1}\left(  x-Ex_{0}\right)  \right. \\
&  \left.  -\frac{1}{4}\left(  x-Ex_{0}\right)  ^{T}C^{-1}A\left(  t\right)
C^{-1}\left(  x-Ex_{0}\right)  \right\}  dx
\end{align*}
and letting again $x=Ex_{0}+2C^{1/2}y$ inside the integral%
\begin{align*}
&  \frac{\partial_{t}\left(  \det C\right)  }{2\det C}=-\frac{1}{\pi^{N/2}%
}\int_{\mathbb{R}^{N}}e^{-\left\vert y\right\vert ^{2}}\cdot\\
&  \cdot\left\{  \left(  x_{0}^{T}E^{T}+2y^{T}C^{1/2}\right)  B^{T}%
C^{-1/2}y-y^{T}C^{-1/2}A\left(  t\right)  C^{-1/2}y\right\}  dy\\
&  =-\frac{1}{\pi^{N/2}}\frac{\pi^{N/2}}{2}\left(  0+2\operatorname{Tr}%
C^{1/2}B^{T}C^{-1/2}+\operatorname{Tr}C^{-1/2}A\left(  t\right)
C^{-1/2}\right) \\
&  =-\operatorname{Tr}C^{1/2}B^{T}C^{-1/2}+\frac{1}{2}\operatorname{Tr}%
C^{-1/2}A\left(  t\right)  C^{-1/2}.
\end{align*}
where we used Lemma \ref{Lemma Gauss xAx}. Finally, since similar matrices
have the same trace,
\begin{align*}
&  -\operatorname{Tr}C^{1/2}BC^{-1/2}+\frac{1}{2}\operatorname{Tr}%
C^{-1/2}A\left(  t\right)  C^{-1/2}\\
&  =-\operatorname{Tr}B+\frac{1}{2}\operatorname{Tr}A\left(  t\right)  C^{-1},
\end{align*}
so we are done.
\end{proof}

The proof of Proposition \ref{Prop Ddet} also completes the proof of Theorem
\ref{Thm LG=0}.

\begin{remark}
\label{Remark stima t fissato}Since, by Theorem \ref{Thm LG=0}, we can write
\[
\partial_{t}\Gamma\left(  x,t,x_{0},t_{0}\right)  =\sum_{i,j=1}^{q}%
a_{ij}\left(  t\right)  \partial_{x_{i}x_{j}}^{2}\Gamma\left(  x,t,x_{0}%
,t_{0}\right)  +\sum_{k,j=1}^{N}b_{jk}x_{k}\partial_{x_{j}}\Gamma\left(
x,t,x_{0},t_{0}\right)  ,
\]
the function $\partial_{t}\Gamma$ satisfies upper bounds analogous to those
proved in Theorem \ref{Thm bounds derivatives} for $\partial_{x_{i}x_{j}}%
^{2}\Gamma$.
\end{remark}

Let us now show that $\Gamma$ satisfies, with respect to the other variables,
the transposed equation, that is:

\begin{theorem}
\label{Thm L*Gamma=0}Letting%
\[
\mathcal{L}^{\ast}u=\sum_{i,j=1}^{q}a_{ij}\left(  s\right)  \partial
_{y_{i}y_{j}}^{2}u-\sum_{k,j=1}^{N}b_{jk}y_{k}\partial_{j_{j}}%
u-u\operatorname*{Tr}B+\partial_{s}u
\]
we have, for every fixed $\left(  x,t\right)  $%
\[
\mathcal{L}^{\ast}\left(  \Gamma\left(  x,t;\cdot,\cdot\right)  \right)
\left(  y,s\right)  =0
\]
for a.e. $s<t$ and every $y$.
\end{theorem}

\begin{proof}
We keep the notation used in the proof of Proposition \ref{Prop comput der}:%
\begin{align*}
C^{\prime}\left(  t,s\right)   &  =E\left(  t-s\right)  ^{T}C\left(
t,s\right)  ^{-1}E\left(  t-s\right) \\
\Gamma\left(  x,t;y,s\right)   &  =\frac{1}{\left(  4\pi\right)  ^{N/2}%
\sqrt{\det C\left(  t,s\right)  }}e^{-\left(  \frac{1}{4}\left(  y-E\left(
s-t\right)  x\right)  ^{T}C^{\prime}\left(  t,s\right)  \left(  y-E\left(
s-t\right)  x\right)  +\left(  t-s\right)  \operatorname*{Tr}B\right)  }.
\end{align*}
Exploiting (\ref{der 1 y}) and (\ref{der 2 y}) we have, by a tedious
computation which is analogous to that in the proof of Theorem \ref{Thm LG=0}%
,
\begin{align*}
\mathcal{L}^{\ast}\Gamma\left(  x,t;y,s\right)   &  =\frac{1}{2}\Gamma\left(
x,t;y,s\right)  \left\{  -\operatorname{Tr}A\left(  s\right)  C^{\prime
}\left(  t,s\right)  -\frac{\partial_{s}\left(  \det C\left(  t,s\right)
\right)  }{\det C\left(  t,s\right)  }\right. \\
&  +y^{T}B^{T}C^{\prime}\left(  t,s\right)  y-y^{T}B^{T}E\left(  t-s\right)
^{T}C\left(  t,s\right)  ^{-1}x\\
&  \left.  +\left(  BE\left(  t-s\right)  y\right)  ^{T}C\left(  t,s\right)
^{-1}\left(  x-E\left(  t-s\right)  y\right)  \right\} \\
&  =\frac{1}{2}\Gamma\left(  x,t;y,s\right)  \left\{  -\operatorname{Tr}%
A\left(  s\right)  C^{\prime}\left(  t,s\right)  -\frac{\partial_{s}\left(
\det C\left(  t,s\right)  \right)  }{\det C\left(  t,s\right)  }\right\}  .
\end{align*}
So we are done provided that:
\end{proof}

\begin{proposition}
\label{Prop Ddet2}For a.e. $s<t$ we have%
\[
\frac{\partial_{s}\left(  \det C\left(  t,s\right)  \right)  }{2\det C\left(
t,s\right)  }=-\operatorname*{Tr}A\left(  s\right)  C^{\prime}\left(
t,s\right)  .
\]

\end{proposition}

\begin{proof}
Taking $\frac{\partial}{\partial s}$ in the identity (\ref{int Gamma tr}) we
have, by (\ref{D_tGlast}), for almost every $s<t,$
\begin{align*}
&  e^{-\left(  t-s\right)  \operatorname*{Tr}B}\operatorname*{Tr}%
B=\int_{\mathbb{R}^{N}}\frac{\partial\Gamma}{\partial s}\left(  x,t;x_{0}%
,s\right)  dx\\
&  =-\int_{\mathbb{R}^{N}}\Gamma\left(  x,t;x_{0},s\right)  \cdot\\
&  \cdot\left\{  \frac{\partial_{s}\left(  \det C\right)  }{2\det
C}-\operatorname*{Tr}B-\frac{1}{2}\left(  BE\left(  t-s\right)  x_{0}\right)
^{T}C\left(  t,s\right)  ^{-1}\left(  x-E\left(  t-s\right)  x_{0}\right)
\right. \\
&  \left.  +\frac{1}{4}\left(  E\left(  s-t\right)  x-x_{0}\right)
^{T}C^{\prime}\left(  t,s\right)  A\left(  s\right)  C^{\prime}\left(
t,s\right)  \left(  E\left(  s-t\right)  x-x_{0}\right)  \right\}  dx
\end{align*}%
\begin{align*}
&  =-\left\{  \frac{\partial_{s}\left(  \det C\right)  }{2\det C}%
-\operatorname*{Tr}B\right\}  e^{-\left(  t-s\right)  \operatorname*{Tr}B}\\
&  -\int_{\mathbb{R}^{N}}\Gamma\left(  x,t;x_{0},s\right)  \left\{  -\frac
{1}{2}\left(  BE\left(  t-s\right)  x_{0}\right)  ^{T}C\left(  t,s\right)
^{-1}\left(  x-E\left(  t-s\right)  x_{0}\right)  \right. \\
&  \left.  +\frac{1}{4}\left(  E\left(  s-t\right)  x-x_{0}\right)
^{T}C^{\prime}\left(  t,s\right)  A\left(  s\right)  C^{\prime}\left(
t,s\right)  \left(  E\left(  s-t\right)  x-x_{0}\right)  \right\}  dx
\end{align*}
hence%
\begin{align*}
&  \frac{\partial_{s}\left(  \det C\right)  }{2\det C}=-\frac{1}{\left(
4\pi\right)  ^{N/2}\sqrt{\det C}}\int_{\mathbb{R}^{N}}e^{-\frac{1}{4}\left(
x-E\left(  t-s\right)  x_{0}\right)  ^{T}C\left(  t,s\right)  ^{-1}\left(
x-E\left(  t-s\right)  x_{0}\right)  }\cdot\\
&  \cdot\left\{  -\frac{1}{2}\left(  BE\left(  t-s\right)  x_{0}\right)
^{T}C\left(  t,s\right)  ^{-1}\left(  x-E\left(  t-s\right)  x_{0}\right)
\right. \\
&  \left.  +\frac{1}{4}\left(  E\left(  s-t\right)  x-x_{0}\right)
^{T}C^{\prime}\left(  t,s\right)  A\left(  s\right)  C^{\prime}\left(
t,s\right)  \left(  E\left(  s-t\right)  x-x_{0}\right)  \right\}  dx
\end{align*}
and letting again $x=E\left(  t-s\right)  x_{0}+2C^{1/2}\left(  t,s\right)  y$
inside the integral, applying Lemma \ref{Lemma Gauss xAx} and
(\ref{null Gaussian}), with some computation we get%
\[
\frac{\partial_{s}\left(  \det C\right)  }{\det C}=-\operatorname{Tr}%
C^{-1/2}\left(  t,s\right)  E\left(  t-s\right)  A\left(  s\right)  E\left(
t-s\right)  ^{T}C\left(  t,s\right)  ^{-1/2}.
\]
Since $C^{-1/2}\left(  t,s\right)  E\left(  t-s\right)  A\left(  s\right)
E\left(  t-s\right)  ^{T}C\left(  t,s\right)  ^{-1/2}$ and $A\left(  s\right)
C^{\prime}\left(  t,s\right)  $ are similar, they have the same trace, so the
proof is concluded.
\end{proof}

\subsection{The Cauchy problem\label{sec Cauchy}}

In this section we will prove points (iv), (v), (vii) of Theorem
\ref{Thm main}.

We are going to show that the Cauchy problem can be solved, by means of our
fundamental solution $\Gamma$. Just to simplify notation, let us now take
$t_{0}=0$ and let $C\left(  t\right)  =C\left(  t,0\right)  $. We have the following:

\begin{theorem}
\label{Thm Cauchy}Let
\begin{align}
u\left(  x,t\right)   &  =\int_{\mathbb{R}^{N}}\Gamma\left(  x,t;y,0\right)
f\left(  y\right)  dy\nonumber\\
&  =\frac{e^{-t\operatorname*{Tr}B}}{\left(  4\pi\right)  ^{N/2}\sqrt{\det
C\left(  t\right)  }}\int_{\mathbb{R}^{N}}e^{-\frac{1}{4}\left(  x-E\left(
t\right)  y\right)  ^{T}C\left(  t\right)  ^{-1}\left(  x-E\left(  t\right)
y\right)  }f\left(  y\right)  dy. \label{repr formula Cauchy1}%
\end{align}
Then:

(a) if $f\in L^{p}\left(  \mathbb{R}^{N}\right)  $ for some $p\in\left[
1,\infty\right]  $ or $f\in C_{b}^{0}\left(  \mathbb{R}^{N}\right)  $ (bounded
continuous) then $u$ solves the equation $\mathcal{L}u=0$ in $\mathbb{R}%
^{N}\times\left(  0,\infty\right)  $ and $u\left(  \cdot,t\right)  \in
C^{\infty}\left(  \mathbb{R}^{N}\right)  $ for every fixed $t>0$.

(b) if $f\in C^{0}\left(  \mathbb{R}^{N}\right)  $ and there exists $C>0$ such
that (\ref{exp bound}) holds, then there exists $T>0$ such that $u$ solves the
equation $\mathcal{L}u=0$ in $\mathbb{R}^{N}\times\left(  0,T\right)  $ and
$u\left(  \cdot,t\right)  \in C^{\infty}\left(  \mathbb{R}^{N}\right)  $ for
every fixed $t\in\left(  0,T\right)  $.

The initial condition $f$ is attained in the following senses:

(i) For every $p\in\lbrack1,+\infty),$ if $f\in L^{p}\left(  \mathbb{R}%
^{N}\right)  $ we have $u\left(  \cdot,t\right)  \in L^{p}\left(
\mathbb{R}^{N}\right)  $ for every $t>0$, and%
\[
\left\Vert u\left(  \cdot,t\right)  -f\right\Vert _{L^{p}\left(
\mathbb{R}^{N}\right)  }\rightarrow0\text{ as }t\rightarrow0^{+}.
\]

(ii) If $f\in L^{\infty}\left(  \mathbb{R}^{N}\right)  $ and $f$ is continuous
at some point $x_{0}\in\mathbb{R}^{N}$ then%
\[
u\left(  x,t\right)  \rightarrow f\left(  x_{0}\right)  \text{ as }\left(
x,t\right)  \rightarrow\left(  x_{0},0\right)  .
\]

(iii) If $f\in C_{\ast}^{0}\left(  \mathbb{R}^{N}\right)  $ (i.e., vanishing
at infinity) then%
\[
\sup_{x\in\mathbb{R}^{N}}\left\vert u\left(  x,t\right)  -f\left(  x\right)
\right\vert \rightarrow0\text{ as }t\rightarrow0^{+}.
\]

(iv) If $f\in C^{0}\left(  \mathbb{R}^{N}\right)  $ and satisfies
(\ref{exp bound}), then%
\[
u\left(  x,t\right)  \rightarrow f\left(  x_{0}\right)  \text{ as }\left(
x,t\right)  \rightarrow\left(  x_{0},0\right)  .
\]

\end{theorem}

\begin{proof}
From Theorem \ref{Thm bounds derivatives}, (i), we read that for $\left(
x,t\right)  $ ranging in a compact subset of $\mathbb{R}^{N}\times\left(
0,+\infty\right)  $, and every $y\in\mathbb{R}^{N}$,%
\[
\sum_{\left\vert \alpha\right\vert \leq n}\left\vert \partial_{x}^{\alpha
}\Gamma\left(  x,t;y,0\right)  \right\vert \leq ce^{-c_{1}\left\vert
y\right\vert ^{2}}\cdot\left\{  1+\left\vert y\right\vert ^{n}\right\}
\]
for suitable constants $c,c_{1}>0$. Moreover, by Remark
\ref{Remark stima t fissato}, $\left\vert \partial_{t}\Gamma\right\vert $ also
satisfies this bound (with $n=2$). This implies that for every $f\in
L^{p}\left(  \mathbb{R}^{N}\right)  $ for some $p\in\left[  1,\infty\right]
$, (in particular for $f\in C_{b}^{0}\left(  \mathbb{R}^{N}\right)  $) the
integral defining $u$ converges and $\mathcal{L}u$ can be computed taking the
derivatives inside the integral. Moreover, all the derivatives $u_{x_{i}%
},u_{x_{i}x_{j}}$ are continuous, while $u_{t}$ is defined only almost
everywhere, and locally essentially bounded. Then by Theorem \ref{Thm LG=0} we
have $\mathcal{L}u\left(  x,t\right)  =0$ for a.e. $t>0$ and every
$x\in\mathbb{R}^{N}$. Also, the $x$-derivatives of every order can be actually
taken under the integral sign, so that $u\left(  \cdot,t\right)  \in
C^{\infty}\left(  \mathbb{R}^{N}\right)  $. This proves (a). Postponing for a
moment the proof of (b), to show that $u$ attains the initial condition
(points (i)-(iii)) let us perform, inside the integral in
(\ref{repr formula Cauchy1}), the change of variables
\begin{align*}
C\left(  t\right)  ^{-1/2}\left(  x-E\left(  t\right)  y\right)   &  =2z\\
y  &  =E\left(  -t\right)  \left(  x-2C\left(  t\right)  ^{1/2}z\right) \\
dy  &  =2^{N}e^{t\operatorname*{Tr}B}\det C\left(  t\right)  ^{1/2}dz
\end{align*}
so that%
\[
u\left(  x,t\right)  =\frac{1}{\pi^{N/2}}\int_{\mathbb{R}^{N}}e^{-\left\vert
z\right\vert ^{2}}f\left(  E\left(  -t\right)  \left(  x-2C\left(  t\right)
^{1/2}z\right)  \right)  dz
\]
and, since $\int_{\mathbb{R}^{N}}\frac{e^{-\left\vert z\right\vert ^{2}}}%
{\pi^{N/2}}dz=1,$%
\[
\left\vert u\left(  x,t\right)  -f\left(  x\right)  \right\vert \leq
\int_{\mathbb{R}^{N}}\frac{e^{-\left\vert z\right\vert ^{2}}}{\pi^{N/2}%
}\left\vert f\left(  E\left(  -t\right)  \left(  x-2C\left(  t\right)
^{1/2}z\right)  \right)  -f\left(  x\right)  \right\vert dz.
\]
Let us now proceed separately in the three cases.

(i) By Minkowsky's inequality for integrals we have%
\[
\left\Vert u\left(  \cdot,t\right)  -f\right\Vert _{L^{p}\left(
\mathbb{R}^{N}\right)  }\leq\int_{\mathbb{R}^{N}}\frac{e^{-\left\vert
z\right\vert ^{2}}}{\pi^{N/2}}\left\Vert f\left(  E\left(  -t\right)  \left(
\cdot-2C\left(  t\right)  ^{1/2}z\right)  \right)  -f\left(  \cdot\right)
\right\Vert _{L^{p}\left(  \mathbb{R}^{N}\right)  }dz.
\]
Next,%
\begin{align*}
&  \left\Vert f\left(  E\left(  -t\right)  \left(  \cdot-2C\left(  t\right)
^{1/2}z\right)  \right)  -f\left(  \cdot\right)  \right\Vert _{L^{p}\left(
\mathbb{R}^{N}\right)  }\\
&  \leq\left\Vert f\left(  E\left(  -t\right)  \left(  \cdot-2C\left(
t\right)  ^{1/2}z\right)  \right)  \right\Vert _{L^{p}\left(  \mathbb{R}%
^{N}\right)  }+\left\Vert f\right\Vert _{L^{p}\left(  \mathbb{R}^{N}\right)
}\\
&  =\left\Vert f\left(  E\left(  -t\right)  \left(  \cdot\right)  \right)
\right\Vert _{L^{p}\left(  \mathbb{R}^{N}\right)  }+\left\Vert f\right\Vert
_{L^{p}\left(  \mathbb{R}^{N}\right)  }\leq c\left\Vert f\right\Vert
_{L^{p}\left(  \mathbb{R}^{N}\right)  }%
\end{align*}
for $0<t<1$, since%
\[
\left\Vert f\left(  E\left(  -t\right)  \left(  \cdot\right)  \right)
\right\Vert _{L^{p}\left(  \mathbb{R}^{N}\right)  }^{p}=\int_{\mathbb{R}^{N}%
}\left\vert f\left(  E\left(  -t\right)  \left(  x\right)  \right)
\right\vert ^{p}dx
\]
letting $E\left(  -t\right)  x=y;x=E\left(  t\right)
y;dx=e^{-t\operatorname*{Tr}B}dy,$%
\[
=e^{-t\operatorname*{Tr}B}\left\Vert f\right\Vert _{L^{p}\left(
\mathbb{R}^{N}\right)  }\leq e^{\left\vert \operatorname*{Tr}B\right\vert
}\left\Vert f\right\Vert _{L^{p}\left(  \mathbb{R}^{N}\right)  }\text{ for
}0<t<1.
\]
This means that for every $t\in\left(  0,1\right)  $ we have%
\[
\frac{e^{-\left\vert z\right\vert ^{2}}}{\pi^{N/2}}\left\Vert f\left(
E\left(  -t\right)  \left(  \cdot-2C\left(  t\right)  ^{1/2}z\right)  \right)
-f\left(  \cdot\right)  \right\Vert _{L^{p}\left(  \mathbb{R}^{N}\right)
}\leq c\left\Vert f\right\Vert _{L^{p}\left(  \mathbb{R}^{N}\right)  }%
\frac{e^{-\left\vert z\right\vert ^{2}}}{\pi^{N/2}}\in L^{1}\left(
\mathbb{R}^{N}\right)  .
\]
Let us show that for a.e. fixed $z\in\mathbb{R}^{N}$ we also have
\[
\frac{e^{-\left\vert z\right\vert ^{2}}}{\pi^{N/2}}\left\Vert f\left(
E\left(  -t\right)  \left(  \cdot-2C\left(  t\right)  ^{1/2}z\right)  \right)
-f\left(  \cdot\right)  \right\Vert _{L^{p}\left(  \mathbb{R}^{N}\right)
}\rightarrow0\text{ as }t\rightarrow0^{+},
\]
this will imply the desired result by Lebesgue's theorem.%
\begin{align*}
&  \left\Vert f\left(  E\left(  -t\right)  \left(  \cdot-2C\left(  t\right)
^{1/2}z\right)  \right)  -f\left(  \cdot\right)  \right\Vert _{L^{p}\left(
\mathbb{R}^{N}\right)  }\\
&  \leq\left\Vert f\left(  E\left(  -t\right)  \left(  \cdot-2C\left(
t\right)  ^{1/2}z\right)  \right)  -f\left(  E\left(  -t\right)  \cdot\right)
\right\Vert _{L^{p}\left(  \mathbb{R}^{N}\right)  }+\left\Vert f\left(
E\left(  -t\right)  \left(  \cdot\right)  \right)  -f\right\Vert
_{L^{p}\left(  \mathbb{R}^{N}\right)  }.
\end{align*}
Now:%
\begin{align*}
&  \left\Vert f\left(  E\left(  -t\right)  \left(  \cdot-2C\left(  t\right)
^{1/2}z\right)  \right)  -f\left(  E\left(  -t\right)  \cdot\right)
\right\Vert _{L^{p}\left(  \mathbb{R}^{N}\right)  }^{p}\\
&  =\int_{\mathbb{R}^{N}}\left\vert f\left(  E\left(  -t\right)  \left(
x-2C\left(  t\right)  ^{1/2}z\right)  \right)  -f\left(  E\left(  -t\right)
x\right)  \right\vert ^{p}dx\\
&  =e^{t\operatorname*{Tr}B}\int_{\mathbb{R}^{N}}\left\vert f\left(
y-2E\left(  -t\right)  C\left(  t\right)  ^{1/2}z\right)  -f\left(  y\right)
\right\vert ^{p}dy\rightarrow0
\end{align*}
for $z$ fixed and $t\rightarrow0^{+}$, because $2E\left(  -t\right)  C\left(
t\right)  ^{1/2}z\rightarrow0$ and the translation operator is continuous on
$L^{p}\left(  \mathbb{R}^{N}\right)  .$

It remains to show that
\[
\left\Vert f\left(  E\left(  -t\right)  \left(  \cdot\right)  \right)
-f\right\Vert _{L^{p}\left(  \mathbb{R}^{N}\right)  }\rightarrow0\text{ as
}t\rightarrow0^{+}\text{,}%
\]
which is not straightforward. For every fixed $\varepsilon>0$, let $\phi$ be a
compactly supported continous function such that $\left\Vert f-\phi\right\Vert
_{p}<\varepsilon$, then%
\begin{align*}
\left\Vert f\left(  E\left(  -t\right)  \left(  \cdot\right)  \right)
-f\right\Vert _{p}  &  \leq\left\Vert f\left(  E\left(  -t\right)  \left(
\cdot\right)  \right)  -\phi\left(  E\left(  -t\right)  \left(  \cdot\right)
\right)  \right\Vert _{p}\\
&  +\left\Vert \phi\left(  E\left(  -t\right)  \left(  \cdot\right)  \right)
-\phi\right\Vert _{p}+\left\Vert f-\phi\right\Vert _{p}%
\end{align*}
and%
\[
\left\Vert f\left(  E\left(  -t\right)  \left(  \cdot\right)  \right)
-\phi\left(  E\left(  -t\right)  \left(  \cdot\right)  \right)  \right\Vert
_{p}=\left(  e^{t\operatorname*{Tr}B}\right)  ^{1/p}\left\Vert f-\phi
\right\Vert _{p}\leq\left(  e^{\left\vert \operatorname*{Tr}B\right\vert
}\right)  ^{1/p}\varepsilon
\]
for $t\in\left(  0,1\right)  $. Let $\operatorname*{sprt}\phi\subset
B_{R}\left(  0\right)  $, then for every $t\in\left(  0,1\right)  $ we have
$\left\vert E\left(  -t\right)  \left(  x\right)  \right\vert \leq c\left\vert
x\right\vert $ so that
\[
\left\Vert \phi\left(  E\left(  -t\right)  \left(  \cdot\right)  \right)
-\phi\right\Vert _{L^{p}\left(  \mathbb{R}^{N}\right)  }^{p}=\int_{\left\vert
x\right\vert <CR}\left\vert \phi\left(  E\left(  -t\right)  \left(  x\right)
\right)  -\phi\left(  x\right)  \right\vert ^{p}dx.
\]
Since for every $x\in\mathbb{R}^{N}$, $\phi\left(  E\left(  -t\right)  \left(
x\right)  \right)  \rightarrow\phi\left(  x\right)  $ as $t\rightarrow0^{+}$
and
\[
\left\vert \phi\left(  E\left(  -t\right)  \left(  x\right)  \right)
-\phi\left(  x\right)  \right\vert ^{p}\leq2\max\left\vert \phi\right\vert
^{p}%
\]
which is integrable on $B_{CR}\left(  0\right)  $, by uniform continuity of
$\phi$,%
\[
\left\Vert \phi\left(  E\left(  -t\right)  \left(  \cdot\right)  \right)
-\phi\right\Vert _{L^{p}\left(  \mathbb{R}^{N}\right)  }\rightarrow0\text{ as
}t\rightarrow0^{+},
\]
hence for $t$ small enough%
\[
\left\Vert f\left(  E\left(  -t\right)  \left(  \cdot\right)  \right)
-f\right\Vert _{p}\leq c\varepsilon,
\]
and we are done.

(ii) Let $f\in L^{\infty}\left(  \mathbb{R}^{N}\right)  $, and let $f$ be
continuous at some point $x_{0}\in\mathbb{R}^{N}$ then%
\[
\left\vert u\left(  x,t\right)  -f\left(  x_{0}\right)  \right\vert \leq
\int_{\mathbb{R}^{N}}\frac{e^{-\left\vert z\right\vert ^{2}}}{\pi^{N/2}%
}\left\vert f\left(  E\left(  -t\right)  \left(  x-2C\left(  t\right)
^{1/2}z\right)  \right)  -f\left(  x_{0}\right)  \right\vert dz.
\]
Now, for fixed $z\in\mathbb{R}^{N}$ and $\left(  x,t\right)  \rightarrow
\left(  x_{0},0\right)  $ we have%
\begin{align*}
E\left(  -t\right)  \left(  x-2C\left(  t\right)  ^{1/2}z\right)   &
\rightarrow x_{0}\\
f\left(  E\left(  -t\right)  \left(  x-2C\left(  t\right)  ^{1/2}z\right)
\right)   &  \rightarrow f\left(  x_{0}\right)
\end{align*}
while%
\[
\frac{e^{-\left\vert z\right\vert ^{2}}}{\pi^{N/2}}\left\vert f\left(
E\left(  -t\right)  \left(  x-2C\left(  t\right)  ^{1/2}z\right)  \right)
-f\left(  x_{0}\right)  \right\vert \leq2\left\Vert f\right\Vert _{L^{\infty
}\left(  \mathbb{R}^{N}\right)  }\frac{e^{-\left\vert z\right\vert ^{2}}}%
{\pi^{N/2}}\in L^{1}\left(  \mathbb{R}^{N}\right)
\]
hence by Lebesgue's theorem%
\[
\left\vert u\left(  x,t\right)  -f\left(  x_{0}\right)  \right\vert
\rightarrow0.
\]

(iii) As in point (i) we have%
\[
\left\Vert u\left(  \cdot,t\right)  -f\right\Vert _{L^{\infty}\left(
\mathbb{R}^{N}\right)  }\leq\int_{\mathbb{R}^{N}}\frac{e^{-\left\vert
z\right\vert ^{2}}}{\pi^{N/2}}\left\Vert f\left(  E\left(  -t\right)  \left(
\cdot-2C\left(  t\right)  ^{1/2}z\right)  \right)  -f\left(  \cdot\right)
\right\Vert _{L^{\infty}\left(  \mathbb{R}^{N}\right)  }dz
\]
and as in point (ii)
\begin{align*}
&  \frac{e^{-\left\vert z\right\vert ^{2}}}{\pi^{N/2}}\left\Vert f\left(
E\left(  -t\right)  \left(  \cdot-2C\left(  t\right)  ^{1/2}z\right)  \right)
-f\left(  \cdot\right)  \right\Vert _{L^{\infty}\left(  \mathbb{R}^{N}\right)
}\\
&  \leq2\left\Vert f\right\Vert _{L^{\infty}\left(  \mathbb{R}^{N}\right)
}\frac{e^{-\left\vert z\right\vert ^{2}}}{\pi^{N/2}}\in L^{1}\left(
\mathbb{R}^{N}\right)  .
\end{align*}
Let us show that for every fixed $z$ we have%
\[
\left\Vert f\left(  E\left(  -t\right)  \left(  \cdot-2C\left(  t\right)
^{1/2}z\right)  \right)  -f\left(  \cdot\right)  \right\Vert _{L^{\infty
}\left(  \mathbb{R}^{N}\right)  }\rightarrow0\text{ as }t\rightarrow0^{+},
\]
hence by Lebesgue's theorem we will conclude the desired assertion.

For every $\varepsilon>0$ we can pick $\phi\in C_{c}^{0}\left(  \mathbb{R}%
^{N}\right)  $ such that $\left\Vert f-\phi\right\Vert _{\infty}<\varepsilon$,
then%
\begin{align*}
&  \left\Vert f\left(  E\left(  -t\right)  \left(  \cdot\right)  \right)
-f\right\Vert _{\infty}\\
&  \leq\left\Vert f\left(  E\left(  -t\right)  \left(  \cdot\right)  \right)
-\phi\left(  E\left(  -t\right)  \left(  \cdot\right)  \right)  \right\Vert
_{\infty}+\left\Vert \phi\left(  E\left(  -t\right)  \left(  \cdot\right)
\right)  -\phi\right\Vert _{\infty}+\left\Vert f-\phi\right\Vert _{\infty}\\
&  <2\varepsilon+\left\Vert \phi\left(  E\left(  -t\right)  \left(
\cdot\right)  \right)  -\phi\right\Vert _{\infty}.
\end{align*}
Since $\phi\ $is compactly supported, there exists $R>0$ such that for every
$t\in\left(  0,1\right)  $ we have $\phi\left(  E\left(  -t\right)  \left(
x\right)  \right)  -\phi\left(  x\right)  \neq0$ only if $\left\vert
x\right\vert <R$.
\[
\left\vert E\left(  -t\right)  \left(  x\right)  -x\right\vert \leq\left\vert
E\left(  -t\right)  -I\right\vert R.
\]
Since $\phi$ is uniformly continuous, for every $\varepsilon>0$ there exists
$\delta>0$ such that for $0<t<\delta$ we have
\[
\left\vert \phi\left(  E\left(  -t\right)  \left(  x\right)  \right)
-\phi\left(  x\right)  \right\vert <\varepsilon
\]
whenever $\left\vert x\right\vert <R.$ So we are done.

Let us now prove (b). To show that $u$ is well defined, smooth in $x$, and
satisfies the equation, for $\left\vert x\right\vert \leq R$ let us write
\begin{align*}
u\left(  x,t\right)   &  =\int_{\left\vert y\right\vert <2R}\Gamma\left(
x,t;y,0\right)  f\left(  y\right)  dy+\int_{\left\vert y\right\vert >2R}%
\Gamma\left(  x,t;y,0\right)  f\left(  y\right)  dy\\
&  \equiv I\left(  x,t\right)  +II\left(  x,t\right)  .
\end{align*}
Since $f$ is bounded for $\left\vert y\right\vert <2R$, reasoning like in the
proof of point (a) we see that $\mathcal{L}I\left(  x,t\right)  $ can be
computed taking the derivatives under the integral sign, so that
$\mathcal{L}I\left(  x,t\right)  =0.$ Moreover, the function $x\mapsto
I\left(  x,t\right)  $ is $C^{\infty}\left(  \mathbb{R}^{N}\right)  $.

To prove the analogous properties for $II\left(  x,t\right)  $ we have to
apply Theorem \ref{Thm bounds derivatives}, (ii): there exists $\delta
\in\left(  0,1\right)  ,C,c>0$ such that for $0<t<\delta$ and every
$x,y\in\mathbb{R}^{N}$ we have, for $n=0,1,2,...$%
\begin{equation}
\sum_{\left\vert \alpha\right\vert \leq n}\left\vert \partial_{x}^{\alpha
}\Gamma\left(  x,t;y,0\right)  \right\vert \leq\frac{C}{t^{Q/2}}%
e^{-c\frac{\left\vert x-E\left(  t\right)  y\right\vert ^{2}}{t}}\cdot\left\{
t^{-\sigma_{N}}+t^{-n\sigma_{N}}\left\vert x-E\left(  t\right)  y\right\vert
^{n}\right\}  .\nonumber
\end{equation}
Recall that $\left\vert x\right\vert <R$ and $\left\vert y\right\vert >2R$.
For $\delta$ small enough and $t\in\left(  \frac{\delta}{2},\delta\right)  $
we have%
\begin{equation}
\sum_{\left\vert \alpha\right\vert \leq n}\left\vert \partial_{x}^{\alpha
}\Gamma\left(  x,t;y,0\right)  \right\vert \leq Ce^{-c\frac{\left\vert
y\right\vert ^{2}}{t}}\cdot\left\{  1+\left\vert y\right\vert ^{n}\right\}
\nonumber
\end{equation}
with constants depending on $\delta,n$. Therefore, if $\alpha$ is the constant
appearing in the assumption (\ref{exp bound}),%
\begin{align*}
&  \int_{\left\vert y\right\vert >2R}\sum_{\left\vert \alpha\right\vert \leq
n}\left\vert \partial_{x}^{\alpha}\Gamma\left(  x,t;y,0\right)  \right\vert
\left\vert f\left(  y\right)  \right\vert dy\\
&  \leq C\int_{\left\vert y\right\vert >2R}e^{-c\frac{\left\vert y\right\vert
^{2}}{\delta}}\cdot\left\{  1+\left\vert y\right\vert ^{n}\right\}
e^{\alpha\left\vert y\right\vert ^{2}}\left\vert f\left(  y\right)
\right\vert e^{-\alpha\left\vert y\right\vert ^{2}}dy\\
&  \leq C\sup_{y\in\mathbb{R}^{N}}\left(  e^{\left(  -\frac{c}{\delta}%
+\alpha\right)  \left\vert y\right\vert ^{2}}\left\{  1+\left\vert
y\right\vert ^{n}\right\}  \right)  \cdot\int_{\mathbb{R}^{N}}\left\vert
f\left(  y\right)  \right\vert e^{-\alpha\left\vert y\right\vert ^{2}}dy
\end{align*}
which shows that for $\delta$ small enough $\mathcal{L}II\left(  x,t\right)  $
can be computed taking the derivatives under the integral sign, so that
$\mathcal{L}II\left(  x,t\right)  =0.$ Moreover, the function $x\mapsto
II\left(  x,t\right)  $ is $C^{\infty}\left(  \mathbb{R}^{N}\right)  $. This
proves (b).

(iv). For $\left\vert x_{0}\right\vert \leq R$ let us write
\[
u\left(  x,t\right)  =\int_{\left\vert y\right\vert <2R}\Gamma\left(
x,t;y,0\right)  f\left(  y\right)  dy+\int_{\left\vert y\right\vert >2R}%
\Gamma\left(  x,t;y,0\right)  f\left(  y\right)  dy\equiv I+II.
\]
Applying point (ii) to $f\left(  y\right)  \chi_{B_{2r}\left(  0\right)  }$ we
have
\[
I=\int_{\left\vert y\right\vert <2R}\Gamma\left(  x,t;y,0\right)  f\left(
y\right)  dy\rightarrow f\left(  x_{0}\right)
\]
as $\left(  x,t\right)  \rightarrow\left(  x_{0},0\right)  $. Let us show that
$II\rightarrow0.$ By (\ref{Upper Gamma t piccolo}) we have%
\[
\left\vert II\right\vert \leq\int_{\left\vert y\right\vert >2R}\frac
{1}{ct^{Q/2}}e^{-c\frac{\left\vert x-E\left(  t\right)  y\right\vert ^{2}}{t}%
}\left\vert f\left(  y\right)  \right\vert dy.
\]
For $y$ fixed with $\left\vert y\right\vert >2R$, hence $\left\vert
x_{0}-y\right\vert \neq0$, we have%
\[
\lim_{\left(  x,t\right)  \rightarrow\left(  x_{0},0\right)  }\frac{1}%
{t^{Q/2}}e^{-c\frac{\left\vert x-E\left(  t\right)  y\right\vert ^{2}}{t}%
}=\lim_{\left(  x,t\right)  \rightarrow\left(  x_{0},0\right)  }\frac
{1}{t^{Q/2}}e^{-c\frac{\left\vert x_{0}-y\right\vert ^{2}}{t}}=0.
\]
Since $\left\vert y\right\vert >2R,$ $\left\vert x_{0}\right\vert <R,$ for
$x\rightarrow x_{0}$ we can assume $\left\vert x\right\vert <\frac{3}{2}R$ and
for $t$ small enough we have $\left\vert x-E\left(  t\right)  y\right\vert
\geq c\left\vert y\right\vert $ for some $c>0$, hence%
\begin{align*}
\frac{1}{ct^{Q/2}}e^{-c\frac{\left\vert x-E\left(  t\right)  y\right\vert
^{2}}{t}}\left\vert f\left(  y\right)  \right\vert \chi_{\left\{  \left\vert
y\right\vert >2R\right\}  }  &  \leq\frac{1}{ct^{Q/2}}e^{-c_{1}\frac
{\left\vert y\right\vert ^{2}}{t}}e^{\alpha\left\vert y\right\vert ^{2}}%
\chi_{\left\{  \left\vert y\right\vert >2R\right\}  }\left\vert f\left(
y\right)  \right\vert e^{-\alpha\left\vert y\right\vert ^{2}}\\
&  \leq\frac{1}{ct^{Q/2}}e^{\left(  \alpha-\frac{c_{1}}{t}\right)  \left\vert
y\right\vert ^{2}}\chi_{\left\{  \left\vert y\right\vert >2R\right\}
}\left\{  \left\vert f\left(  y\right)  \right\vert e^{-\alpha\left\vert
y\right\vert ^{2}}\right\}
\end{align*}
for $t$ small enough%
\begin{align*}
&  \leq\frac{1}{ct^{Q/2}}e^{-\frac{c_{1}}{2t}\left\vert y\right\vert ^{2}}%
\chi_{\left\{  \left\vert y\right\vert >2R\right\}  }\left\{  \left\vert
f\left(  y\right)  \right\vert e^{-\alpha\left\vert y\right\vert ^{2}}\right\}
\\
&  \leq\frac{1}{ct^{Q/2}}e^{-\frac{2c_{1}}{t}R^{2}}\left\{  \left\vert
f\left(  y\right)  \right\vert e^{-\alpha\left\vert y\right\vert ^{2}%
}\right\}  \leq c\left\vert f\left(  y\right)  \right\vert e^{-\alpha
\left\vert y\right\vert ^{2}}\in L^{1}\left(  \mathbb{R}^{N}\right)  .
\end{align*}
Hence by Lebesgue's theorem $II\rightarrow0$ as $\left(  x,t\right)
\rightarrow\left(  x_{0},0\right)  ,$ and we are done.
\end{proof}

\begin{remark}
If $f$ is an unbounded continuous function satisfying (\ref{exp bound}), the
solution of the Cauchy problem can blow up in finite time, already for the
heat operator: the solution of%
\[
\left\{
\begin{array}
[c]{l}%
u_{t}-u_{xx}=0\text{ in }\mathbb{R}\times\left(  0,+\infty\right) \\
u\left(  x,0\right)  =e^{x^{2}}%
\end{array}
\right.
\]
is given by%
\[
u\left(  x,t\right)  =\frac{1}{\sqrt{4\pi t}}\int_{\mathbb{R}}e^{-\frac
{\left(  x-y\right)  ^{2}}{4t}}e^{y^{2}}dy=\frac{e^{\frac{x^{2}}{1-4t}}}%
{\sqrt{1-4t}}\text{ for }0<t<\frac{1}{4},
\]
with $u\left(  x,t\right)  \rightarrow+\infty$ for $t\rightarrow\left(
\frac{1}{4}\right)  ^{-}.$
\end{remark}

\bigskip

We next prove a uniqueness results for the Cauchy problem (\ref{PdC}). In the
following we consider solutions defined in some possibly bounded time interval
$[0,T).$

\begin{theorem}
[Uniqueness]\label{Thm uniqueness} Let $\mathcal{L}$ be an operator of the
form (\ref{L}) satisfying the assumptions (H1)-(H2), let $T\in(0,+\infty]$,
and let either $f\in C(\mathbb{R}^{N})$, or $f\in L^{p}(\mathbb{R}^{N})$ with
$1\leq p<+\infty$.

If $u_{1}$ and $u_{2}$ are two solutions to the same Cauchy problem%
\begin{equation}
\left\{
\begin{array}
[c]{l}%
\mathcal{L}u=0\text{ in }\mathbb{R}^{N}\times\left(  0,T\right)  ,\\
u\left(  \cdot,0\right)  =f,
\end{array}
\right.  \label{PdCg}%
\end{equation}
satisfying (\ref{cond uniqueness}) for some $C>0$, then $u_{1}\equiv u_{2}$ in
$\mathbb{R}^{N}\times\left(  0,T\right)  $.
\end{theorem}

\begin{proof}
Because of the linearity of $\mathcal{L}$, it is enough to prove that if the
function $u:=u_{1}-u_{2}$ satisfies (\ref{PdCg}) with $f=0$ and
(\ref{cond uniqueness}), then $u(x,t)=0$ for every $(x,t)\in\mathbb{R}%
\times\left(  0,+\infty\right)  $. We will prove that $u=0$ in a suitably thin
strip $\mathbb{R}\times\left(  0,t_{1}\right)  $, where $t_{1}$ only depends
on $\mathcal{L}$ and $C$, the assertion then will follow by iterating this argument.

Let $t_{1}\in(0,T]$ be a fixed mumber that will be specified later. For every
positive $R$ we consider a function $h_{R}\in C^{\infty}(\mathbb{R}^{N})$,
such that $h_{R}\left(  \xi\right)  =1$ whenever $\left\vert \xi\right\vert
\leq R$, $h_{R}\left(  \xi\right)  =0$ for every $\left\vert \xi\right\vert
\geq R+1/2$ and that $0\leq h_{R}\left(  \xi\right)  \leq1$. We also assume
that all the first and second order derivatives of $h_{R}$ are bounded by a
constant that doesn't depend on $R$. We fix a point $(y,s)\in\mathbb{R}%
^{N}\times\left(  0,t_{1}\right)  $, and we let $v$ denote the function
\[
v\left(  \xi,\tau\right)  :=h_{R}\left(  \xi\right)  \Gamma\left(
y,s;\xi,\tau\right)  .
\]
For $\varepsilon\in\left(  0,t_{1}/2\right)  $ we define the domain
\[
Q_{R,\varepsilon}:=\left\{  (\xi,\tau)\in\mathbb{R}^{N}\times\left(
0,t_{1}\right)  :\left\vert \xi\right\vert <R+1,\varepsilon<\tau
<s-\varepsilon\right\}
\]
and we also let $Q_{R}=Q_{R,0}$. Note that in $Q_{R,\varepsilon}$ the function
$v\left(  \xi,\tau\right)  $ is smooth in $\xi$ and Lipschitz continuous in
$\tau$.

By (\ref{L}) and (\ref{L star}) we can compute the following Green identity,
with $u$ and $v$ as above.
\begin{align*}
&  v\mathcal{L}u-u\mathcal{L}^{\ast}v\\
&  =\sum_{i,j=1}^{q}a_{ij}\left(  t\right)  \left(  v\partial_{x_{i}x_{j}}%
^{2}u-u\partial_{x_{i}x_{j}}^{2}v\right)  +\sum_{k,j=1}^{N}b_{jk}x_{k}\left(
v\partial_{x_{j}}u+u\partial_{x_{j}}v\right) \\
&  -\left(  v\partial_{t}u+u\partial_{t}v\right)  +uv\operatorname*{Tr}B\\
&  =\sum_{i,j=1}^{q}\partial_{x_{i}}\left(  a_{ij}\left(  t\right)  \left(
v\partial_{x_{j}}u-u\partial_{x_{j}}v\right)  \right)  +\sum_{k,j=1}%
^{N}\partial_{x_{j}}\left(  b_{jk}x_{k}uv\right)  -\partial_{t}\left(
uv\right)  .
\end{align*}
We now integrate the above identity on $Q_{R,\varepsilon}$ and apply the
divergence theorem, noting that $v,\partial_{x_{1}}v,\dots,\partial_{x_{N}}v$
vanish on the lateral part of the boundary of $Q_{R,\varepsilon}$, by the
properties of $h_{R}$. Hence:%
\begin{equation}%
\begin{split}
&  \int_{Q_{R,\varepsilon}}v(\xi,\tau)\mathcal{L}u(\xi,\tau)-u(\xi
,\tau)\mathcal{L}^{\ast}v(\xi,\tau)d\xi\,d\tau\\
&  =\int_{\mathbb{R}^{N}}u(\xi,\varepsilon)v(\xi,\varepsilon)d\xi
-\int_{\mathbb{R}^{N}}u(\xi,s-\varepsilon)v(\xi,s-\varepsilon)d\xi.
\end{split}
\label{Green-identity}%
\end{equation}
Concerning the last integral, since the function $y\mapsto h_{R}(y)u(y,s)$ is
continuous and compactly supported, by Theorem \ref{Thm Cauchy}, (iii) we have
that
\[
\int_{\mathbb{R}^{N}}u(\xi,s-\varepsilon)v(\xi,s-\varepsilon)d\xi
=\int_{\mathbb{R}^{N}}u(\xi,s-\varepsilon)h_{R}(\xi)\Gamma(y,s;\xi
,s-\varepsilon)d\xi\rightarrow h_{R}(y)u(y,s)
\]
as $\varepsilon\rightarrow0^{+}$. Moreover
\[
\int_{\mathbb{R}^{N}}u(\xi,\varepsilon)v(\xi,\varepsilon)d\xi=\int%
_{\mathbb{R}^{N}}u(\xi,\varepsilon)h_{R}(\xi)\Gamma(y,s;\xi,\varepsilon
)d\xi\rightarrow0,
\]
as $\varepsilon\rightarrow0^{+}$, since $\Gamma$ is a bounded function
whenever $(\xi,\varepsilon)\in\mathbb{R}^{N}\times\left(  0,s/2\right)  $, and
$u(\cdot,\varepsilon)h_{R}\rightarrow0$ either uniformly, if the inital datum
is assumed by continuity, or in the $L^{p}$ norm. Using the fact that
$\mathcal{L}u=0$ and $u(\cdot,0)=0$, we conclude that, as $\left\vert
y\right\vert <R$, (\ref{Green-identity}) gives%
\begin{equation}
u(y,s)=\int_{Q_{R}}u(\xi,\tau)\mathcal{L}^{\ast}v(\xi,\tau)d\xi\,d\tau.
\label{Green-identity-2}%
\end{equation}
Since $\mathcal{L}^{\ast}\Gamma(y,s;\xi,\tau)=0$ whenever $\tau<s$, we have%
\begin{align*}
&  \mathcal{L}^{\ast}\left(  h_{R}\Gamma\right)  =\sum_{i,j=1}^{q}%
a_{ij}\left(  \tau\right)  \partial_{\xi_{i}\xi_{j}}^{2}\left(  h_{R}%
\Gamma\right)  -\sum_{k,j=1}^{N}b_{jk}\xi_{k}\partial_{\xi_{j}}\left(
h_{R}\Gamma\right)  -h_{R}\left(  \Gamma\operatorname*{Tr}B+\partial_{\tau
}\Gamma\right) \\
&  =\Gamma\sum_{i,j=1}^{q}a_{ij}\left(  \tau\right)  \partial_{\xi_{i}\xi_{j}%
}^{2}h_{R}+2\sum_{i,j=1}^{q}a_{ij}\left(  \tau\right)  \left(  \partial
_{\xi_{i}}h_{R}\right)  \left(  \partial_{\xi_{j}}\Gamma\right)  -\Gamma
\sum_{k,j=1}^{N}b_{jk}\xi_{k}\partial_{\xi_{j}}h_{R}%
\end{align*}
therefore the identity (\ref{Green-identity-2}) yields, since $\partial
_{\xi_{i}}h_{R}=0$ for $\left\vert \xi\right\vert \leq R$,%
\begin{equation}%
\begin{split}
u(y,s)=  &  \int_{Q_{R}\backslash Q_{R-1}}u(\xi,\tau)\left\{  \sum_{i,j=1}%
^{q}a_{i,j}(\tau)\cdot\right. \\
\cdot &  \left[  2\partial_{\xi_{i}}h_{R}(\xi)\partial_{\xi_{j}}\Gamma\left(
y,s;\xi,\tau\right)  +\Gamma(y,s;\xi,\tau)\partial_{\xi_{i}\xi_{j}}h_{R}%
(\xi)\right] \\
&  \left.  -\sum_{k,j=1}^{N}b_{jk}\xi_{k}\partial_{\xi_{j}}h_{R}(\xi
)\Gamma(y,s;\xi,\tau)\right\}  d\xi\,d\tau.
\end{split}
\label{Green-identity-3}%
\end{equation}
We claim that (\ref{Green-identity-3}) implies
\begin{equation}
\left\vert u(y,s)\right\vert \leq\int_{Q_{R}\backslash Q_{R-1}}C_{1}%
|u(\xi,\tau)|e^{-C|\xi|^{2}}d\xi\,d\tau, \label{unici Claim}%
\end{equation}
for some positive constant $C_{1}$ only depending on the operator
$\mathcal{L}$ and on the uniform bound f the derivatives of $h_{R}$, provided
that $t_{1}$ is sufficiently small. Our assertion then follows by letting
$R\rightarrow+\infty$.

So we are left to prove (\ref{unici Claim}). By Proposition \ref{Prop lim t0}
we know that, for suitable constants $\delta\in\left(  0,1\right)  $,
$c_{1},c_{2}>0$, for $0<s-\tau\leq\delta$ and every $y,\xi\in\mathbb{R}^{N}$
we have:%
\begin{equation}
\Gamma\left(  y,s,\xi,\tau\right)  \leq\frac{c_{1}}{\left(  s-\tau\right)
^{Q/2}}e^{-c_{2}\frac{\left\vert y-E\left(  s-\tau\right)  \xi\right\vert
^{2}}{s-\tau}}. \label{unici 1}%
\end{equation}
Moreover, from the computation in the proof of Theorem \ref{Thm L*Gamma=0} we
read that%
\[
\nabla_{\xi}\Gamma\left(  y,s;\xi,\tau\right)  =-\frac{1}{2}\Gamma\left(
y,s;\xi,\tau\right)  C^{\prime}\left(  s,\tau\right)  \left(  \xi-E\left(
\tau-s\right)  y\right)
\]
where
\[
C^{\prime}\left(  s,\tau\right)  =E\left(  s-\tau\right)  ^{T}C\left(
s,\tau\right)  ^{-1}E\left(  s-\tau\right)  .
\]
Hence%
\[
\nabla_{\xi}\Gamma\left(  y,s;\xi,\tau\right)  =\frac{1}{2}\Gamma\left(
y,s;\xi,\tau\right)  E\left(  s-\tau\right)  ^{T}C\left(  s,\tau\right)
^{-1}\left(  y-E\left(  s-\tau\right)  \xi\right)  .
\]
By (\ref{fq inverse}) we have inequality for matrix norms
\[
\left\Vert C\left(  s,\tau\right)  ^{-1}\right\Vert \leq c\left\Vert
C_{0}\left(  s-\tau\right)  ^{-1}\right\Vert
\]
and, for $0<s-\tau\leq\delta$
\[
\leq c\left\Vert C_{0}^{\ast\left(  s-\tau\right)  -1}\right\Vert \leq
c\left\Vert D_{0}\left(  s-\tau\right)  \right\Vert ^{-1}%
\]
hence%
\begin{align}
\left\vert \nabla_{\xi}\Gamma\left(  y,s;\xi,\tau\right)  \right\vert  &  \leq
c\Gamma\left(  y,s;\xi,\tau\right)  \left\Vert D_{0}\left(  s-\tau\right)
\right\Vert ^{-1}\left\vert y-E\left(  s-\tau\right)  \xi\right\vert
\nonumber\\
&  \leq\frac{c_{1}}{\left(  s-\tau\right)  ^{\frac{Q}{2}+\sigma_{N}}}%
e^{-c_{2}\frac{\left\vert y-E\left(  s-\tau\right)  \xi\right\vert ^{2}%
}{s-\tau}}\left\vert y-E\left(  s-\tau\right)  \xi\right\vert .
\label{unici 2}%
\end{align}
Now, in the integral in (\ref{Green-identity-3}) we have $R<\left\vert
\xi\right\vert <R+1$. Then for $\left\vert y\right\vert <R/2$ and
$0<s-\tau\leq\delta<1$ we have%
\begin{align*}
\frac{\left\vert \xi\right\vert }{2}  &  \leq\left\vert \xi\right\vert
-\left\vert y\right\vert \leq\left\vert y-\xi\right\vert \leq\left\vert
y-E\left(  s-\tau\right)  \xi\right\vert +\left\vert E\left(  s-\tau\right)
\xi-\xi\right\vert \\
&  \leq\left\vert y-E\left(  s-\tau\right)  \xi\right\vert +\left\Vert
E\left(  s-\tau\right)  -I\right\Vert \left\vert \xi\right\vert \leq\left\vert
y-E\left(  s-\tau\right)  \xi\right\vert +\frac{\left\vert \xi\right\vert }%
{4}.
\end{align*}
Hence%
\[
\left\vert y-E\left(  s-\tau\right)  \xi\right\vert \geq\frac{\left\vert
\xi\right\vert }{4}.
\]
Moreover%
\[
\left\vert y-E\left(  s-\tau\right)  \xi\right\vert \leq\left\vert
y\right\vert +c\left\vert \xi\right\vert \leq c_{1}\left\vert \xi\right\vert
.
\]
Hence (\ref{unici 1})-(\ref{unici 2}) give%
\begin{align*}
\Gamma\left(  y,s,\xi,\tau\right)   &  \leq\frac{c_{1}}{\left(  s-\tau\right)
^{Q/2}}e^{-c_{3}\frac{\left\vert \xi\right\vert ^{2}}{s-\tau}}\\
\left\vert \partial_{\xi_{j}}\Gamma\left(  y,s;\xi,\tau\right)  \right\vert
&  \leq\frac{c_{1}}{\left(  s-\tau\right)  ^{\frac{Q}{2}+\sigma_{N}}%
}\left\vert \xi\right\vert e^{-c_{3}\frac{\left\vert \xi\right\vert ^{2}%
}{s-\tau}}.
\end{align*}

Therefore (\ref{Green-identity-3}) gives%
\[
\left\vert u(y,s)\right\vert \leq\int_{Q_{R}\backslash Q_{R-1}}|u(\xi
,\tau)|\left\{  \frac{c_{1}}{\left(  s-\tau\right)  ^{\frac{Q}{2}+\sigma_{N}}%
}\left\vert \xi\right\vert e^{-c_{3}\frac{\left\vert \xi\right\vert ^{2}%
}{s-\tau}}\right\}  d\xi\,d\tau.
\]
We can assume $R>1,$ writing, for $0<s-\tau<1$ and every $\xi\in\mathbb{R}%
^{N}$ with $\left\vert \xi\right\vert >1$,%
\begin{align*}
&  \frac{c_{1}}{\left(  s-\tau\right)  ^{\frac{Q}{2}+\sigma_{N}}}\left\vert
\xi\right\vert e^{-c_{3}\frac{\left\vert \xi\right\vert ^{2}}{s-\tau}}%
=\frac{c_{1}}{\left(  s-\tau\right)  ^{\frac{Q}{2}+\sigma_{N}}}\left\vert
\xi\right\vert e^{-c_{3}\frac{1}{s-\tau}}e^{-c_{3}\frac{\left\vert
\xi\right\vert ^{2}-1}{s-\tau}}\\
&  \leq c\left\vert \xi\right\vert e^{-c_{3}\frac{\left\vert \xi\right\vert
^{2}-1}{s-\tau}}\leq c\left\vert \xi\right\vert e^{-c_{3}\left(  \left\vert
\xi\right\vert ^{2}-1\right)  }=c_{4}\left\vert \xi\right\vert e^{-c_{3}%
\left\vert \xi\right\vert ^{2}}\leq c_{5}e^{-c_{6}\left\vert \xi\right\vert
^{2}}.
\end{align*}
This implies the Claim, so we are done.
\end{proof}

The link between the existence result of Theorem \ref{Thm Cauchy} and the
uniqueness result of Theorem \ref{Thm uniqueness} is completed by the following

\begin{proposition}
\label{Prop cond uniqueness}(a) Let $f$ be a bounded continuous function on
$\mathbb{R}^{N}$, or a function belonging to $L^{p}\left(  \mathbb{R}%
^{N}\right)  $ for some $p\in\lbrack1,\infty)$.\ Then the function
\[
u\left(  x,t\right)  =\int_{\mathbb{R}^{N}}\Gamma\left(  x,t;y,0\right)
f\left(  y\right)  dy
\]
satisfies the condition (\ref{cond uniqueness}) for every fixed constants
$T,C>0.$

(b) If $f\in C^{0}\left(  \mathbb{R}^{N}\right)  $ satisfies the condition
(\ref{exp bound}) for some constant $\alpha>0$ then the function $u$ satisfies
(\ref{cond uniqueness}) for some $T,C>0.$
\end{proposition}

This means that in the class of functions satisfying (\ref{cond uniqueness})
there exists one and only one solution to the Cauchy problem, under any of the
above assumptions on the initial datum $f$.

\bigskip

\begin{proof}
(a) If $f$ is bounded continuous we simply have%
\[
\left\vert u\left(  x,t\right)  \right\vert \leq\left\Vert f\right\Vert
_{C_{b}^{0}\left(  \mathbb{R}^{N}\right)  }\int_{\mathbb{R}^{N}}\Gamma\left(
x,t;y,0\right)  dy=\left\Vert f\right\Vert _{C_{b}^{0}\left(  \mathbb{R}%
^{N}\right)  }%
\]
by Proposition \ref{Prop integral Gamma dx}. Hence (\ref{cond uniqueness})
holds for every fixed $C,T>0.$

Let now $f\in L^{p}\left(  \mathbb{R}^{N}\right)  $ for some $p\in
\lbrack1,\infty)$. Let us write%
\begin{align*}
u\left(  x,t\right)   &  =\frac{e^{-t\operatorname*{Tr}B}}{\left(
4\pi\right)  ^{N/2}\sqrt{\det C\left(  t\right)  }}\int_{\mathbb{R}^{N}%
}e^{-\frac{1}{4}\left(  x-E\left(  t\right)  y\right)  ^{T}C\left(  t\right)
^{-1}\left(  x-E\left(  t\right)  y\right)  }f\left(  y\right)  dy\\
&  =\frac{e^{-t\operatorname*{Tr}B}}{\left(  4\pi\right)  ^{N/2}\sqrt{\det
C\left(  t\right)  }}\int_{\mathbb{R}^{N}}e^{-\frac{1}{4}\left(  E\left(
-t\right)  x-y\right)  ^{T}C^{\prime}\left(  t\right)  \left(  E\left(
-t\right)  x-y\right)  }f\left(  y\right)  dy\\
&  =\frac{e^{-t\operatorname*{Tr}B}}{\left(  4\pi\right)  ^{N/2}\sqrt{\det
C\left(  t\right)  }}\left(  k_{t}\ast f\right)  \left(  E\left(  -t\right)
x\right)
\end{align*}
having set%
\[
k_{t}\left(  x\right)  =e^{-\frac{1}{4}x^{T}C^{\prime}\left(  t\right)  x}.
\]
Then
\begin{align}
&  \int_{0}^{T}\left(  \int_{\mathbb{R}^{N}}\left\vert u\left(  x,t\right)
\right\vert e^{-C\left\vert x\right\vert ^{2}}dx\right)  dt\nonumber\\
&  \leq\int_{0}^{T}\frac{e^{-t\operatorname*{Tr}B}}{\left(  4\pi\right)
^{N/2}\sqrt{\det C\left(  t\right)  }}\left(  \int_{\mathbb{R}^{N}}\left\vert
\left(  k_{t}\ast f\right)  \left(  E\left(  -t\right)  x\right)  \right\vert
e^{-C\left\vert x\right\vert ^{2}}dx\right)  dt. \label{ineq 1}%
\end{align}
Applying H\"{o}lder inequality with $q^{-1}+p^{-1}=1$ and Young's inequality
we get:%
\begin{align}
&  \int_{\mathbb{R}^{N}}\left\vert \left(  k_{t}\ast f\right)  \left(
E\left(  -t\right)  x\right)  \right\vert e^{-C\left\vert x\right\vert ^{2}%
}dx\nonumber\\
E\left(  -t\right)  x  &  =y;x=E\left(  t\right)  y;dx=e^{-t\operatorname{Tr}%
B}dy\nonumber\\
&  =e^{-t\operatorname{Tr}B}\int_{\mathbb{R}^{N}}\left\vert \left(  k_{t}\ast
f\right)  \left(  y\right)  \right\vert e^{-C\left\vert E\left(  t\right)
y\right\vert ^{2}}dy\nonumber\\
&  \leq e^{-t\operatorname{Tr}B}\left\Vert k_{t}\ast f\right\Vert
_{L^{p}\left(  \mathbb{R}^{N}\right)  }\left\Vert e^{-C\left\vert E\left(
t\right)  y\right\vert ^{2}}\right\Vert _{L^{q}\left(  \mathbb{R}^{N}\right)
}\nonumber\\
&  \leq c\left(  q,T\right)  e^{-t\operatorname{Tr}B}\left\Vert f\right\Vert
_{L^{p}\left(  \mathbb{R}^{N}\right)  }\left\Vert k_{t}\right\Vert
_{L^{1}\left(  \mathbb{R}^{N}\right)  } \label{ineq 2}%
\end{align}
and inserting (\ref{ineq 2}) into (\ref{ineq 1}) we have%
\begin{align*}
&  \int_{0}^{T}\left(  \int_{\mathbb{R}^{N}}\left\vert u\left(  x,t\right)
\right\vert e^{-C\left\vert x\right\vert ^{2}}dx\right)  dt\\
&  \leq\int_{0}^{T}\frac{e^{-t\operatorname*{Tr}B}}{\left(  4\pi\right)
^{N/2}\sqrt{\det C\left(  t\right)  }}c\left(  q,T\right)
e^{-t\operatorname{Tr}B}\left\Vert f\right\Vert _{L^{p}\left(  \mathbb{R}%
^{N}\right)  }\int_{\mathbb{R}^{N}}e^{-\frac{1}{4}x^{T}C^{\prime}\left(
t\right)  x}dxdt\\
&  =c\left(  q,T\right)  \left\Vert f\right\Vert _{L^{p}\left(  \mathbb{R}%
^{N}\right)  }\int_{0}^{T}\int_{\mathbb{R}^{N}}\frac{e^{-t\operatorname*{Tr}%
B}}{\left(  4\pi\right)  ^{N/2}\sqrt{\det C\left(  t\right)  }}%
e^{-t\operatorname{Tr}B}e^{-\frac{1}{4}x^{T}C^{\prime}\left(  t\right)
x}dxdt\\
x  &  =E\left(  -t\right)  w;dx=e^{t\operatorname{Tr}B}dw\\
&  =c\left(  q,T\right)  \left\Vert f\right\Vert _{L^{p}\left(  \mathbb{R}%
^{N}\right)  }\int_{0}^{T}\int_{\mathbb{R}^{N}}\frac{e^{-t\operatorname*{Tr}%
B}}{\left(  4\pi\right)  ^{N/2}\sqrt{\det C\left(  t\right)  }}e^{-\frac{1}%
{4}\left(  E\left(  -t\right)  w\right)  ^{T}C^{\prime}\left(  t\right)
E\left(  -t\right)  w}dwdt\\
&  =c\left(  q,T\right)  \left\Vert f\right\Vert _{L^{p}\left(  \mathbb{R}%
^{N}\right)  }\int_{0}^{T}\int_{\mathbb{R}^{N}}\Gamma\left(  w,t;0,0\right)
dwdt\\
&  =c\left(  q,T\right)  \left\Vert f\right\Vert _{L^{p}\left(  \mathbb{R}%
^{N}\right)  }\int_{0}^{T}e^{-t\operatorname*{Tr}B}dt\leq c\left(  q,T\right)
\left\Vert f\right\Vert _{L^{p}\left(  \mathbb{R}^{N}\right)  }%
\end{align*}
by (\ref{int Gamma tr}). Hence (\ref{cond uniqueness}) still holds for every
fixed $C,T>0.$

(b) Assume that
\[
\int_{\mathbb{R}^{N}}\left\vert f\left(  y\right)  \right\vert e^{-\alpha
\left\vert y\right\vert ^{2}}dy<\infty
\]
for some $\alpha>0$ and, for $T\in\left(  0,1\right)  ,\beta>0$ to be chosen
later, let us bound:%
\begin{align*}
&  \int_{0}^{T}\left(  \int_{\mathbb{R}^{N}}\left\vert u\left(  x,t\right)
\right\vert e^{-\beta\left\vert x\right\vert ^{2}}dx\right)  dt\\
&  \leq\int_{0}^{T}\left(  \int_{\mathbb{R}^{N}}\left(  \frac
{e^{-t\operatorname*{Tr}B}}{\left(  4\pi\right)  ^{N/2}\sqrt{\det C\left(
t\right)  }}\int_{\mathbb{R}^{N}}e^{-\frac{1}{4}\left(  x-E\left(  t\right)
y\right)  ^{T}C\left(  t\right)  ^{-1}\left(  x-E\left(  t\right)  y\right)
}\left\vert f\left(  y\right)  \right\vert dy\right)  e^{-\beta\left\vert
x\right\vert ^{2}}dx\right)  dt
\end{align*}%
\[
y=E\left(  -t\right)  \left(  x-2C\left(  t\right)  ^{1/2}z\right)
;dy=e^{t\operatorname*{Tr}B}2^{N}\det C\left(  t\right)  ^{1/2}dz
\]%
\[
=\int_{0}^{T}\int_{\mathbb{R}^{N}}\frac{e^{-\left\vert z\right\vert ^{2}}}%
{\pi^{N/2}}\left(  \int_{\mathbb{R}^{N}}\left\vert f\left(  E\left(
-t\right)  \left(  x-2C\left(  t\right)  ^{1/2}z\right)  \right)  \right\vert
e^{-\beta\left\vert x\right\vert ^{2}}dx\right)  dzdt
\]%
\[
E\left(  -t\right)  \left(  x-2C\left(  t\right)  ^{1/2}z\right)
=w;e^{t\operatorname{Tr}B}dx=dw
\]%
\begin{align*}
&  =\int_{0}^{T}\int_{\mathbb{R}^{N}}\frac{e^{-\left\vert z\right\vert ^{2}}%
}{\pi^{N/2}}\left(  \int_{\mathbb{R}^{N}}e^{-t\operatorname{Tr}B}\left\vert
f\left(  w\right)  \right\vert e^{-\beta\left\vert E\left(  t\right)
w+2C\left(  t\right)  ^{1/2}z\right\vert ^{2}}dw\right)  dzdt\\
&  =\int_{0}^{T}e^{-t\operatorname{Tr}B}\int_{\mathbb{R}^{N}}\frac
{e^{-\left\vert z\right\vert ^{2}}}{\pi^{N/2}}\cdot\\
&  \cdot\left(  \int_{\mathbb{R}^{N}}\left\vert f\left(  w\right)  \right\vert
e^{-\beta\left(  \left\vert E\left(  t\right)  w\right\vert ^{2}+4\left\vert
C\left(  t\right)  ^{1/2}z\right\vert ^{2}+2\left(  E\left(  t\right)
w\right)  ^{T}C\left(  t\right)  ^{1/2}z\right)  }dw\right)  dzdt\\
&  =\int_{0}^{T}\frac{e^{-t\operatorname{Tr}B}}{\pi^{N/2}}\left(
\int_{\mathbb{R}^{N}}\left\vert f\left(  w\right)  \right\vert e^{-\beta
\left\vert E\left(  t\right)  w\right\vert ^{2}}\right.  \cdot\\
&  \left.  \cdot\left(  \int_{\mathbb{R}^{N}}e^{-\left\vert z\right\vert ^{2}%
}e^{-4\beta\left\vert C\left(  t\right)  ^{1/2}z\right\vert ^{2}}%
e^{-2\beta\left(  E\left(  t\right)  w\right)  ^{T}C\left(  t\right)  ^{1/2}%
z}dz\right)  dw\right)  dt.
\end{align*}
Next, for $0<t<1$ we have, since $\left\Vert C\left(  t\right)  \right\Vert
\leq ct,$%
\[
\left\vert -2\beta\left(  E\left(  t\right)  w\right)  ^{T}C\left(  t\right)
^{1/2}z\right\vert \leq c_{1}\beta\left\vert w\right\vert \sqrt{t}\left\vert
z\right\vert
\]
so that%
\begin{align*}
&  \int_{0}^{T}\left(  \int_{\mathbb{R}^{N}}\left\vert u\left(  x,t\right)
\right\vert e^{-\beta\left\vert x\right\vert ^{2}}dx\right)  dt\\
&  \leq\frac{e^{\left\vert \operatorname{Tr}B\right\vert }}{\pi^{N/2}}\int%
_{0}^{T}\left(  \int_{\mathbb{R}^{N}}\left\vert f\left(  w\right)  \right\vert
e^{-\beta\left\vert E\left(  t\right)  w\right\vert ^{2}}\left(
\int_{\mathbb{R}^{N}}e^{-\left\vert z\right\vert ^{2}}e^{c_{1}\beta\left\vert
w\right\vert \sqrt{t}\left\vert z\right\vert }dz\right)  dw\right)  dt.
\end{align*}
Next,%
\begin{align*}
\int_{\mathbb{R}^{N}}e^{-\left\vert z\right\vert ^{2}}e^{c_{1}\beta\left\vert
w\right\vert \sqrt{t}\left\vert z\right\vert }dz  &  =c_{n}\int_{0}^{+\infty
}e^{-\rho^{2}+c_{1}\beta\left\vert w\right\vert \sqrt{t}\rho}\rho^{n-1}d\rho\\
&  \leq c\int_{0}^{+\infty}e^{-\frac{\rho^{2}}{2}+c_{1}\beta\rho\sqrt{t}}%
d\rho=ce^{c_{2}\beta^{2}t\left\vert w\right\vert ^{2}}%
\end{align*}
and%
\[
\int_{0}^{T}\left(  \int_{\mathbb{R}^{N}}\left\vert u\left(  x,t\right)
\right\vert e^{-\beta\left\vert x\right\vert ^{2}}dx\right)  dt\leq c\int%
_{0}^{T}\left(  \int_{\mathbb{R}^{N}}\left\vert f\left(  w\right)  \right\vert
e^{-\beta\left\vert E\left(  t\right)  w\right\vert ^{2}}e^{c_{2}\beta
^{2}t\left\vert w\right\vert ^{2}}dw\right)  dt.
\]
Since $E\left(  t\right)  $ is invertible and $E\left(  0\right)  =1$, for $T$
small enough and $t\in(0,T)$ we have $\left\vert E\left(  t\right)
w\right\vert \geq\frac{1}{2}\left\vert w\right\vert $ so that%
\[
e^{-\beta\left\vert E\left(  t\right)  w\right\vert ^{2}}e^{c_{2}\beta
^{2}t\left\vert w\right\vert ^{2}}\leq e^{-\left\vert w\right\vert ^{2}%
\beta\left(  \frac{1}{2}-c_{2}t\beta\right)  }.
\]
We now fix $\beta=4\alpha$ and then fix $T$ small enough such that $\frac
{1}{2}-c_{2}T\beta\geq\frac{1}{4},$ so that for $t\in\left(  0,T\right)  $ we
have%
\[
e^{-\left\vert w\right\vert ^{2}\beta\left(  \frac{1}{2}-c_{2}t\beta\right)
}\leq e^{-\left\vert w\right\vert ^{2}\beta\left(  \frac{1}{2}-c_{2}%
T\beta\right)  }\leq e^{-\alpha\left\vert w\right\vert ^{2}}%
\]
and
\[
\int_{0}^{T}\left(  \int_{\mathbb{R}^{N}}\left\vert u\left(  x,t\right)
\right\vert e^{-\beta\left\vert x\right\vert ^{2}}dx\right)  dt\leq c\int%
_{0}^{T}\left(  \int_{\mathbb{R}^{N}}\left\vert f\left(  w\right)  \right\vert
e^{-\alpha\left\vert w\right\vert ^{2}}dw\right)  dt<\infty.
\]
So we are done.
\end{proof}

The previous uniqueness property for the Cauchy problem also implies the
following replication property for the heat kernel:

\begin{corollary}
\label{Corollary replication}For every $x,y\in\mathbb{R}^{N}$ and $s<\tau<t$
we have%
\[
\Gamma\left(  x,t;y,s\right)  =\int_{\mathbb{R}^{N}}\Gamma\left(
x,t;z,\tau\right)  \Gamma\left(  z,\tau;y,s\right)  dz.
\]

\end{corollary}

\begin{proof}
Let%
\begin{align*}
u\left(  x,t\right)   &  =\int_{\mathbb{R}^{N}}\Gamma\left(  x,t;z,\tau
\right)  \Gamma\left(  z,\tau;y,s\right)  dz\\
f\left(  z\right)   &  =\Gamma\left(  z,\tau;y,s\right)
\end{align*}
for $y\in\mathbb{R}^{N}$ fixed, $\tau>s$ fixed. By Theorem \ref{Thm main},
(i), $f\in C_{\ast}^{0}\left(  \mathbb{R}^{N}\right)  $. Hence by Theorem
\ref{Thm Cauchy}, point (iii), $u$ solves the Cauchy problem%
\[
\left\{
\begin{array}
[c]{l}%
\mathcal{L}u\left(  x,t\right)  =0\text{ for }t>\tau\\
u\left(  x,\tau\right)  =\Gamma\left(  x,\tau;y,s\right)
\end{array}
\right.
\]
where the initial datum is assumed continuously, uniformly as $t\rightarrow
\tau$. Since $v\left(  x,t\right)  =\Gamma\left(  x,t;y,s\right)  $ solves the
same Cauchy problem, by Theorem \ref{Thm uniqueness} the assertion follows.
\end{proof}

\bigskip

\textsc{Marco Bramanti: Dipartimento di Matematica, Politecnico di Milano, Via
Bonardi 9, I-20133 Milano, Italy.}

marco.bramanti@polimi.it

\bigskip

\textsc{Sergio Polidoro: Dipartimento di Scienze Fisiche, Informatiche e
Matematiche, Universita degli Studi di Modena e Reggio Emilia. Via Campi
213/b, I-41121 Modena, Italy.}

sergio.polidoro@unimore.it


\bigskip

\begin{thebibliography}{99}                                                                                               %


\bibitem {AP}F. Anceschi, S. Polidoro: A survey on the classical theory for
Kolmogorov equation. Le Matematiche, vol. LXXV (2020), Issue I, pp.221-258.
doi: 10.4418/2020.75.1.11

\bibitem {BSur}M. Bramanti: An invitation to hypoelliptic operators and
H\"{o}rmander's vector fields. SpringerBriefs in Mathematics. Springer, Cham,
2014. xii+150 pp.

\bibitem {DP}G. Da Prato: Introduction to stochastic differential equations.
Second edition. Appunti dei Corsi Tenuti da Docenti della Scuola. [Notes of
Courses Given by Teachers at the School] Scuola Normale Superiore, Pisa, 1998.

\bibitem {DelarMen}F. Delarue, S. Menozzi: Density estimates for a random
noise propagating through a chain of differential equations. J. Funct. Anal.,
259 (2010), pp. 1577-1630.

\bibitem {DiFraPasc}M. Di Francesco, A. Pascucci: On a class of degenerate
parabolic equations of Kolmogorov type, AMRX Appl. Math. Res. Express, (2005),
pp. 77-116.

\bibitem {DiFraPo}M. Di Francesco, S. Polidoro: Schauder estimates, Harnack
inequality and Gaussian lower bound for Kolmogorov-type operators in
non-divergence form. Adv. Differential Equations, 11 (2006), pp. 1261-1320.

\bibitem {FL}B. Farkas, L. Lorenzi: On a class of hypoelliptic operators with
unbounded coefficients in $\mathbb{R}^{N}$. Comm. on Pure Appl. Anal., 8 (4)
(2009), 1159-1201.

\bibitem {Hor2}L. H\"{o}rmander: Hypoelliptic second order differential
equations. Acta Math. 119 (1967), 147-171.

\bibitem {il}A. M. Il'in: On a class of ultraparabolic equations. (Russian)
Dokl. Akad. Nauk SSSR 159 1964 1214--1217.

\bibitem {Ko}A. N. Kolmogorov: Zuf\"{a}llige Bewegungen (zur Theorie der
Brownschen Bewegung). Ann. of Math. (2) 35 (1934), no. 1, 116--117.

\bibitem {K1}L. P. Kupcov: The fundamental solutions of a certain class of
elliptic-parabolic second order equations. (Russian) Differencial'nye
Uravnenija 8 (1972), 1649--1660, 1716.

\bibitem {K2}L. P. Kuptsov: Fundamental solutions of some second-order
degenerate parabolic equations. (Russian) Mat. Zametki 31 (1982), no. 4,
559--570, 654. English translation: Mathematical Notes of the Academy of
Sciences of the USSR (1982) 31: 283. https://doi.org/10.1007/BF01138938.

\bibitem {LP}E. Lanconelli, S. Polidoro: On a class of hypoelliptic evolution
operators. Partial differential equations, II (Turin, 1993). Rend. Sem. Mat.
Univ. Politec. Torino 52 (1994), no. 1, 29--63.

\bibitem {Lu}A. Lunardi: Schauder estimates for a class of degenerate elliptic
and parabolic operators with unbounded coefficients in $\mathbb{R}^{n}$. Ann.
Scuola Norm. Sup. Pisa (4) 24 (1997), 133-164.

\bibitem {Ma}M. Manfredini: The Dirichlet problem for a class of
ultraparabolic equations. Adv. Differential Equations 2 (1997), no. 5, 831--866.

\bibitem {PaPe}A. Pascucci, A. Pesce: On stochastic Langevin and Fokker-Planck
equations: the two-dimensional case. \emph{Preprint} arXiv:1910.05301

\bibitem {PaPe2}A. Pascucci, A. Pesce: The parametrix method for parabolic
SPDEs. To appear in Stochastic Process. Appl., 2020 https://doi.org/10.1016/j.spa.2020.05.008.

\bibitem {Po}S. Polidoro: On a class of ultraparabolic operators of
Kolmogorov-Fokker-Planck type. Matematiche (Catania) 49 (1994), no. 1, 53--105 (1995)

\bibitem {So}I. M. Sonin: A class of degenerate diffusion processes. (Russian)
Teor. Verojatnost. i Primenen 12 1967 540--547. English Transl., Theory
Probab. Appl. 12 (1967), 490-496.

\bibitem {We}M. Weber: The fundamental solution of a degenerate partial
differential equation of parabolic type. Trans. Amer. Math. Soc. 71 (1951), 24--37.
\end{thebibliography}
\end{document}